\documentclass[a4paper, twoside]{amsart}

\usepackage[a4paper,body={15.6cm,23cm},left=3cm]{geometry}

\usepackage[T1]{fontenc}
\usepackage[latin1]{inputenc}

\usepackage{graphicx}

\usepackage{amsfonts}
\usepackage{amsthm}
\usepackage{mathrsfs}
\usepackage{latexsym}
\usepackage{amssymb}

\usepackage{bbm}
\usepackage{dsfont}

\usepackage[all]{xy}

\usepackage{mdwlist}

\usepackage{hyperref}
\usepackage{ae,aecompl}

\usepackage{bbm}
\usepackage{dsfont}

\newcommand{\N}{\mathbb{N}}    
\newcommand{\Z}{\mathbb{Z}}   
\newcommand{\Q}{\mathbb{Q}}   
   
\newcommand{\C}{\mathbb{C}}   

\newcommand{\V}{\mathbb{V}}

\newcommand{\Oc}{\mathcal{O}}
\newcommand{\Kc}{\mathcal{K}}
\newcommand{\Qc}{\mathcal{Q}}
\newcommand{\Ec}{\mathcal{E}}

\newcommand{\Ac}{\mathcal{A}}

\newcommand{\Bc}{\mathcal{B}}

\newcommand{\p}{\mathfrak{p}}

\newcommand{\aff}[1]{\,\mathbb{A}^{#1}_k}
\newcommand{\proj}[1]{\, \mathbb{P}^{#1}_k}

\newcommand{\opone}{\mathcal{O}_{\mathbb{P}^1_k}}
\newcommand{\pone}{\mathbb{P}^1_k}

\newcommand{\iso}{\stackrel{\cong}{\longrightarrow}}

\newcommand{\vbar}{\, \big| \,}

\newcommand{\wo}{\backslash}

\theoremstyle{plain}
\newtheorem{theorem}{Theorem}[section]
\newtheorem{proposition}[theorem]{Proposition}
\newtheorem{lemma}[theorem]{Lemma}
\newtheorem{corollary}[theorem]{Corollary}

\newtheorem*{theorem*}{Theorem}

\theoremstyle{definition}
\newtheorem{example}[theorem]{Example}
\newtheorem{remark}[theorem]{Remark}

\newtheorem{definition}[theorem]{Definition}

\DeclareMathOperator{\rk}{rk}
\DeclareMathOperator{\im}{im}
\DeclareMathOperator{\reg}{reg}

\DeclareMathOperator{\Grass}{Grass}
\DeclareMathOperator{\Quot}{Quot}
\DeclareMathOperator{\Ext}{Ext}
\DeclareMathOperator{\coker}{coker}

\DeclareMathOperator{\GL}{GL}

\DeclareMathOperator{\Spec}{Spec}
\DeclareMathOperator{\Proj}{Proj}

\DeclareMathOperator{\Hom}{Hom}
\DeclareMathOperator{\End}{End}
\DeclareMathOperator{\Isom}{Isom}

\DeclareMathOperator{\Char}{Char}

\DeclareMathOperator{\Rep}{Rep}

\DeclareMathOperator{\id}{id}
\DeclareMathOperator{\Id}{id}

\newcommand{\tSigma}{\tilde{\Sigma}}
\newcommand{\tS}{\tilde{S}}
\newcommand{\tH}{\tilde{H}}
\newcommand{\tL}{\tilde{L}}

\newcommand{\Pb}{\mathbb{P}}

\newcommand{\comment}[1]{#1}

\begin{document}

\title{Quivers, Geometric Invariant Theory, and Moduli of Linear Dynamical Systems}
\author{Markus Bader}
\email{markus.bader@math.uzh.ch}
\address{Institut für Mathematik, 
Universität Z\"urich, 
Winterthurerstrasse 190,
CH-8057 Z\"urich, Switzerland} 


\begin{abstract}
We use geometric invariant theory and the language of quivers to study compactifications of moduli spaces of linear dynamical systems. A general approach to this problem is presented and applied to two well known cases: We show how both Lomadze's and Helmke's compactification arises naturally as a geometric invariant theory quotient. Both moduli spaces are proven to be smooth projective manifolds. Furthermore, a description of Lomadze's compactification as a Quot scheme is given, whereas Helmke's compactification is shown to be an algebraic Grassmann bundle over a Quot scheme. This gives an algebro-geometric description of both compactifications. As an application, we determine the cohomology ring of Helmke's compactification and prove that the two compactifications are not isomorphic when the number of outputs is positive.  
\end{abstract}

\maketitle  

\setcounter{section}{-1}
\section{Introduction}

In this article, we study actions of products of general linear groups on spaces of matrices. We present general techniques from algebraic geometry that we apply to two concrete examples, namely to two different compactifications of the moduli space of controllable linear dynamical systems. The first section introduces geometric invariant theory and representation theory of quivers in a tutorial way. We explain how this machinery can be used to systematically study the problem of compactifying the moduli space of linear dynamical systems. Two important results are presented and we explain how they can be adapted to cover the cases relevant to control theory.
In the second part of the article, we show how both the Helmke and the Lomadze compactification can be constructed as algebraic varieties using this machinery. We obtain an algebro-geometric description of both compactifications that we use to study and compare both varieties.

Moduli spaces of linear dynamical systems have been introduced to control theory by Kalman \cite{kalman:74} and Hazewinkel \cite{hazewinkel:79}. As algebraic varieties they have been constructed and studied among others by Hazewinkel in \cite{hazewinkel:79}, by Byrnes and Hurt in \cite{bh:79}, by Kalman in \cite{kalman:74} and by Tannenbaum in \cite{tannenbaum:91}. In algebraic geometry the main technique to construct moduli spaces is as quotients of algebraic varieties under algebraic group actions using geometric invariant theory. Let $\tSigma_{n,m,p}$ denote the space of linear dynamical systems 
\begin{equation}\begin{split}\label{lds}
x_{t+1} &= Ax_t + Bu_t \\
y_t &= Cx_t + Du_t
\end{split}\end{equation}
with $n$ states, $m$ inputs, and $p$ outputs. It is a space of matrices $\tSigma_{n,m,p} = k^{n \times m} \times k^{n \times p} \times k^{m \times n} \times k^{m \times p}$, where $k$ is a fixed, algebraically closed field of characteristic $0$. Let the group of invertible $n \times n$ matrices $\GL_n$ act on $\tSigma_{n,m,p}$ by change of basis in the state space $k^n$:
\begin{equation}\label{i1} 
\bigl( g, \left(A,B,C,D \right) \bigr) \mapsto \bigl( gAg^{-1}, gB, Cg^{-1}, D \bigr).
\end{equation}
The controllable systems form a Zariski-open subset which we denote with $\tSigma_{n,m,p}^c \subset \tSigma_{n,m,p}$. Geometric invariant theory provides the means to systematically construct such quotients. It associates with every character of the group $\GL_n$, so in particular with the character $\det: \GL_n \longrightarrow k^\ast$,  an open subset of stable points, and realizes the algebraic quotient $\{ \det-\text{stable points} \}// \GL_n$. Byrnes and Hurt \cite {bh:79} were the first to notice that the $\det$-stable points coincide with the controllable systems and therefore that the moduli space of controllable linear dynamical systems can be realized as the quotient $\Sigma^c_{n,m,p}:=\tSigma_{n,m,p}^c // \GL_n$ using GIT.  This quotient is non-projective. Compactifications have been introduced by several authors, let us mention Helmke \cite{helmke:93}, Lomadze \cite{lomadze:90}, and Rosenthal \cite{rosenthal:92}. 

By a quiver we mean an oriented graph, that is a finite set of vertices together with a finite set of oriented edges between the vertices. To every vertex we assign a dimension, and furthermore we mark a subset of vertices. This data is described by diagrams like the following:
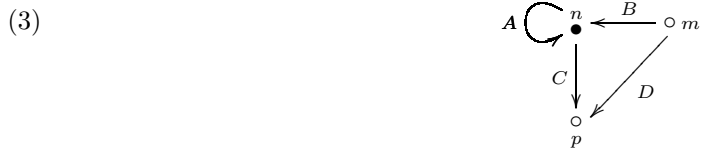
\begin{equation}\label{i2} \xymatrix{ \overset{n}{\bullet \ar@(ul,dl)[]_{A}} \ar[d]_{C} & \circ \, \scriptstyle {m} \ar[l]_{B} \ar[dl]^{D} \\ \underset{p}{\circ} & } \end{equation}
The corresponding GIT problem is the following: we study representations of the quiver of the prescribed dimension. In our concrete example this means matrices $(A,B,C,D) \in \tSigma_{n,m,p}$. The general linear group $\GL_n$ acts by change of basis on the vector space $k^n$ associated with the marked vertex. This corresponds to the group action introduced in (\ref{i1}). In general we will be given a quiver $Q$, a subset of marked vertices $M$, and a dimension vector $v$ (i.e. a prescribed dimension at each vertex). With this data we associate a representation space $\Rep_{Q}^v$ which is always a space of matrices, a group $\GL_{v,M}$ which is always a finite product of general linear groups, and an action of this group on the space of representations.

In that framework, the problem of compactifying the moduli space of linear dynamical systems becomes the following: We are given a space of matrices $\tSigma_{n,m,p}$ which is the space of representations $\tSigma_{n,m,p}=\Rep^v_Q$ for the quiver $Q$ and with dimension vector $v$ as introduced in diagram (\ref{i2}). We mark one vertex - the one corresponding to the state space $k^n$, and we are given a character $\chi=\det$ of the group $\GL_n = \GL_{v,M}$. The quotient space $\{ \det -\text{stable points in } \Rep_{Q}^v \} // \GL_{v,M} = \Sigma_{n,m,p}^c$ is not projective. The goal is to replace the data $(Q,v,M, \chi)$ by a new quiver $\tilde{Q}$, a new dimension vector $\tilde{v}$, a new set of marked vertices $\tilde{M}$, and a new character $\tilde{\chi}$ of $\GL_{\tilde{v}, \tilde{M}}$, such that the quotient $\{ \tilde{\chi}-\text{stable points in } \Rep_{\tilde{Q}}^{\tilde{v}} \}//\GL_{\tilde{v}, \tilde{M}}$ is projective and contains the previous quotient $\Sigma_{n,m,p}^c$ as an open subset. To be more precise, we need to do the following:

\begin{enumerate}
\item Find a new quiver $\tilde{Q}$, a new dimension vector $\tilde{v}$, and a new set of marked vertices $\tilde{M}$, together with morphisms $\varphi: \GL_{Q,M} \longrightarrow \GL_{\tilde{Q}, \tilde{M}}$, $\Phi: \Rep_{Q,v} \longrightarrow \Rep_{\tilde{Q}, \tilde{v}}$, where the latter is a closed embedding, equivariant with respect to $\varphi$.
\item Given a character $\chi$ of $\GL_{v,M}$ (such as $\det: \GL_n \longrightarrow k^\ast$ in our situation), determine the characters $\tilde{\chi}$ of $\GL_{\tilde{v}, \tilde{M}}$, such that under the embedding $\Phi$ the $\chi$-(semi)stable representations will be mapped to the set of $\tilde{\chi}$-(semi)stable representations. Describe the corresponding (semi)stable locus.
\end{enumerate}

In the situations we study, the map $\Phi$ descends to an open embedding of the respective quotient spaces. The first question is: which quivers $\tilde{Q}$ induce projective quotients? This question has been answered by a theorem of Le Bruyn and Procesi \cite{lebruyn:90} in the case where all vertices of the quiver $\tilde{Q}$ are marked. Halic and Stupariu have generalized this result in \cite{halic:04} to arbitrary subsets of marked vertices. It allows the immediate identification of those spaces of matrices that might provide a compactification for the given moduli space. The second task one is confronted with is the identification of the stable and semistable loci corresponding to some character. This is facilitated by a result of King \cite{king:94} which we generalize to our situation.

We apply the strategy outlined above to two compactifications prominent in linear control theory: In the second section, we study  a compactification of the moduli spaces of controllable linear dynamical systems, which is due to Lomadze. He used the following generalization of the equations (\ref{lds}):
\begin{equation}
Kw_{t+1} +Lw_t + M\xi_t=0,
\end{equation}
for matrices $(K,L,M) \in \tL_{n,m,p}:=k^{(n+p) \times n} \times k^{(n+p) \times n} \times  k^{(n+p) \times (n+m)}$. These equations
have been introduced to control theory by Willems (\cite{willems1}, \cite{willems2}, \cite{willems3}, see the book of Kuijper \cite{kuijper:94} for details.). Willems also introduced controllability in a control theoretic way. Lomdaze \cite{lomadze:90} generalized the notion of controllability to the new class of systems by giving the algebraic conditions in Definition \ref{lc}. Denote the set of controllable Lomadze systems with $\tL_{n,m,p}^c$. The Lomadze  compactification  is constructed as a quotient of $\tL_{n,m,p}^c$  under an action of $\GL_{n,n+p}=\GL_{n} \times \GL_{n+p}$. Therefore it can also be obtained using geometric invariant theory and quivers. First, we give an elementary characterization of the possible stability notions on the space of matrices $L_{n,m,p}$ and prove that there are only finitely many different stability conditions. Then we identify those stability conditions that generalize controllability and observability, respectively to the new class of systems. These stability notions are shown to agree with the definitions introduced by Lomadze. As a consequence we can apply the  general theory, which yields a compactification of the corresponding moduli spaces:

\begin{theorem*}A Lomadze system $S=(K,L,M) \in \tL_{n,m,p}$ is controllable if and only if it is stable or semistable with respect to the character $\det^{\chi}$ of $\GL_{n,n+p}$ defined by
\[ {\det}^{\chi} (g_0, g_1) := \det (g_0)^{\chi_0} \det(g_1)^{\chi_1} \text{ for } (g_0,g_1) \in \GL_{n,n+p} \]
for any $\chi=(\chi_0, \chi_1) \in \Z^2$ with $n\chi_0 + (n-1)\chi_1 < 0$ and $\chi_0 + \chi_1 >0$.
The corresponding quotient $L_{n,m,p}^c:= \tL_{n,m,p}^c//\GL_{n,n+p}$ is a smooth projective algebraic variety of dimension $n(m+p) + mp$ and the quotient map is a principal $\GL_{n,n+p}$-bundle. Furthermore, there is a natural open embedding
\[ \Phi_{L} : \Sigma_{n,m,p}^c \longrightarrow L_{n,m,p}^c.\]
\end{theorem*}

Lomadze identified his compactification $L^c_{n,m,p}$ of the moduli space of controllable systems $\Sigma_{n,m,p}^c$ with a variety parametrizing the quotients of the trivial rank $p+m$ vector bundle on $\pone$ that are of rank $p$ and of degree $n$.  Ravi and Rosenthal studied this compactification in  \cite{rosenthal:94} and \cite{rosenthal:95}. They observed that it coincides with the base space of a principal bundle studied extensively by Str\o mme in \cite{stromme:87}. We review it in the end of the second section having two applications in mind: first it allows us to use the results obtained by Str\o mme to better understand the geometry of the variety $L_{n,m,p}^c$. 
Second, the description of $L_{n,m,p}^c$ in terms of Str\o mmes Quot scheme will allow us to give a precise algebro-geometric description of the compactification we study in the third section.

This compactification arises also as a quotient of matrices under an action of some finite product of general linear groups. It has been  introduced by Helmke in \cite{helmke:93}. His construction yields a smooth compact manifold $H_{n,m,p}^c$ containing the manifold $\Sigma_{n,m,p}^c$ of controllable linear dynamical systems as a dense open subset. Furthermore, Helmke showed that there is a  natural map $H_{n,m,p}^c \longrightarrow H_{n,m}^c:=H_{n,m,0}^c$, obtained by forgetting the output, which makes $H_{n,m,p}^c$ into a smooth Grassmann bundle over $H_{n,m}^c$. The starting point for Helmke's compactification is generalizing the equations (\ref{lds}) to \begin{equation} Ex_{t+1} = Ax_t + Bu_t, \; \; Fy_t = Cx_t + Du_t, \end{equation}
by adding matrices $E \in k^{n \times n}$ and $F \in k^{ p \times p}$.
Thus we call a $6$-tuple of matrices $(E,A,B,C,D,F)$ with $E,A \in k^{n \times n}, B \in k^{n \times m}, C \in k^{p \times n}, D \in k^{p \times m}, F \in k^{p \times p}$ a Helmke system. On the space $\tH_{n,m,p}$ of all Helmke systems the group $\GL_{n,n,p}=\GL_n \times \GL_n \times \GL_p$ acts by change of basis. We prove that the notion of controllability introduced by Helmke is a stability notion and identify the corresponding chamber of characters:

\begin{theorem*} A Helmke system $H=(E,A,B,C,D,F)$ is controllable if and only if it is stable or semistable with respect to the character ${\det}^{\chi}$ of the group $\GL_{n,n,p}$, defined by
\[ {\det}^{\chi}(g_0, g_1, g_0):= \det(g_0)^{\chi_0} \det(g_1)^{\chi_1} \det(g_2)^{\chi_2} \text{ for } (g_0, g_1, g_2) \in \GL_{n,n,p} \]
for any $\chi=(\chi_0, \chi_1, \chi_2) \in \Z^3$ with $n\chi_0 + (n-1)\chi_1 + \min \{ p,n\}\chi_2 < 0, \chi_0 + \chi_1 > 0$, and $\chi_2 > 0$. The corresponding quotient $H^c_{n,m,p}:= \tilde{H}_{n,m,p}^c$ is a smooth projective variety of dimension $n(m+p) + mp$ and the quotient map is a principal $\GL_{n,n,p}$-bundle. Furthermore, there is a natural open embedding
\[ \Phi_{H}: \Sigma_{n,m,p}^c \longrightarrow H_{n,m,p}^c. \]
\end{theorem*}

The forgetful map $H_{n,m,p}^c \longrightarrow H_{n,m}^c$ is an algebraic Grassmann bundle. It is immediate that in the case of systems without output Helmke's compactification $H_{n,m}^c$ agrees with Lomadze's compactification $L_{n,m}^c$. Thus we obtain a description of $H_{n,m,p}^c$ as a Grassmann bundle over the Quot scheme studied by Str\o mme. We describe it explicitly by identifying the vector bundle on the base space it is associated with. As an application we use general facts about the cohomology ring of Grassmann bundles to calculate the cohomology ring of $H_{n,m,p}^c$. Also we obtain a precise formula for the rank of the Chow group:

\begin{theorem*}The group underlying the Chow ring $A^\ast(H_{n,m,p}^c)$ is free abelian of rank 
\[ \rk_{\Z} A^\ast(H_{n,m,p}^c) = \binom{n+p+m}{p} \binom{n + 2m -1}{n}. \]
In the case $k=\C$, we have $A^k(H_{n,m,p}^c) \cong H^{2k}(H_{n,m,p}^c,\Z)$ and therefore
\[ \chi_{top} A^\ast(H_{n,m,p}^c) = \rk_{\Z} A^\ast(H_{n,m,p}^c). \]
\end{theorem*}

This allows the comparison of both compactifications with the following result:

\begin{theorem*}For $p>0$, the compactifications $L_{n,m,p}^c$ and  $H_{n,m,p}$ are not isomorphic. If $k=\C$, they are not homeomorphic.
\end{theorem*}

\subsection*{Acknowledgement}
This article is based on my Diplomarbeit completed at the Universit\"at Z\"urich in February, 2005 under the supervision of Ch. Okonek to whom I am grateful for introducing me to this fascinating subject and for his help during my studies. I would also like to thank M. Halic for answering many of my questions on geometric invariant theory and J. Rosenthal for encouraging me to publish this work.

\subsection*{List of key notations}

\begin{description}
\item[$k$] An algebraically closed field of characteristic zero.
\item[$\Oc_X$] The sheaf of regular functions on an algebraic variety $X$.
\item[$\Char(G)$] The character group of an algebraic group $G$.
\item[$\GL_{n_1, \ldots, n_s}$] The algebraic group $\GL_{n_1} \times \ldots \times \GL_{n_s}$.
\item[$\V^{(s)s}(\chi)$] The set of semistable points in an finite dimensional vector space $\V$  that is equipped with an action of an algebraic group $G$, and where $\chi \in \Char(G)$ is a character of $G$.
\item[$X//G$] A categorical quotient of a variety $X$ with respect to the action of an algebraic group $G$.
\item[$Q=(Q_0, Q_1, s,t)$] A quiver with finite sets of vertices $Q_0$ and edges $Q_1$, and with source and target maps $s,t: Q_1 \longrightarrow Q_0$.
\item[$\GL_{v,M}$] Let $Q$ be a quiver, $M \subset Q_0$ a subset of marked vertices, and $v \in \N^{Q_0}$ a dimension vector. Then $\GL_{v,M}= \prod_{i  \in M} \GL_{v_i}$.
\item[$\Rep_Q^v$] The space of representation of a quiver $Q$ of fixed dimension $v \in \N^{Q_0}$.
\item[${k[\V]^G_\chi}$] The vector space of $\chi$-invariant functions of a vector space $\V$, where $\chi$ is a character of an algebraic group $G$ acting on $\V$.
\item[$\Quot_{\pone/k}^{P,\Ec}$] The Quot scheme on $\proj{1}$ parametrizing quotients of $\Ec$ with fixed Hilbert polynomial $P$.
\item[$\Grass(V,d)$] The Grassmannian of quotient spaces of a vector space $V$ of dimension $d$.
\item[$\Grass(n,V)$] The Grassmannian of vector subspaces of a vector space $V$ of dimension $n$.
\item[$A^\ast(X)$] The Chow ring of the algebraic variety $X$.
\end{description}

\begin{description}
\item[$\Sigma_{n,m,p}^c$] The moduli space of controllable linear dynamical systems with $n$ states, $m$ inputs, and $p$ outputs.
\item[$H_{n,m,p}^c$] The Helmke compactification of the moduli space of controllable linear dynamical systems with $n$ states, $m$ inputs, and $p$ outputs.
\item[$L_{n,m,p}^c$] The Lomadze compactification of the moduli space of controllable linear dynamical systems with $n$ states, $m$ inputs, and $p$ outputs.
\end{description}

\section{Moduli of representations of quivers}
 
An elementary introduction to algebraic geometry is \cite{reid:88}. More material is covered in \cite{har:77}. For geometric invariant theory see \cite{pot:97} and \cite{mfk:94}. Quivers are introduced in \cite{auslander-reiten} and \cite{ringel}.
Fix an algebraically closed field $k$ of characteristic zero. We write $\GL_n:= \GL_n(k)$. Every algebraic variety $X$ is equipped with the Zariski-topology and a sheaf of regular functions $\Oc_X$, i.e. for every (Zariski)-open subset $U \subset X$ we are given the ring of regular functions $\Oc_X(U)$. All topological terms refer to the Zariski-topology. 

We start by reviewing some notions of geometric invariant theory applied to the study of representations of quivers. We follow King \cite{king:94}.  Let $G$ be a reductive linear algebraic group. It is in general difficult to construct quotients of algebraic varieties under algebraic group actions. A systematic approach is provided by geometric invariant theory.

\begin{definition}Let $G$ act on an algebraic variety $X$. A categorical quotient of $X$ by $G$ is a variety $Y$, together with a $G$-invariant morphism $\pi: X \longrightarrow Y$, such that the following universal property is satisfied:
Given any other variety $\tilde{Y}$ and any other $G$-invariant morphism $\tilde{\pi}$, there exists a unique morphism $Y \longrightarrow \tilde{Y}$ that makes the following diagram commute:
\[ \xymatrix{  X \ar[d]_{\pi} \ar[dr]^{\tilde{\pi}} & \\ Y \ar@{-->}[r] & \tilde{Y} } \]
\end{definition}

Notice that a categorical quotient is unique up to unique isomorphism. We denote it by $X//G$. Unfortunately, a categorical quotient can be quite far from an orbit space. Consider for an example the $\C^\ast$-action on $\C^n$ given by multiplication. The categorical quotient of this action is $\C^n//\C^\ast = \{ \ast \}$, since every $\C^\ast$-invariant morphism $f: \C^n \longrightarrow Y$ into a variety $Y$ is constant.

\begin{definition}Let $G$ act on an algebraic variety $X$. A pair $(Y, \varphi)$, consisting of a variety $Y$ and a $G$-invariant morphism $\varphi: X \longrightarrow Y$ is called a good quotient of $X$ under the $G$-action, if it verifies the following conditions:
\begin{enumerate}
\item $\varphi$ is affine and surjective;
\item for any open subset $V \subset Y$, the induced morphism 
\[ \varphi^\sharp: \Oc_Y(V) \longrightarrow \Oc_X(\varphi^{-1}(V))^G:= \{ f : \varphi^{-1}(V) \longrightarrow k \vbar f \text{ is regular and } G\text{-invariant} \}\]
 is an isomorphism of rings;
\item Any two disjoint, $G$-invariant, closed subsets in $X$ have disjoint and closed images in $Y$. 
\end{enumerate}

\noindent A good quotient $(Y, \varphi)$ is called a geometric quotient if it induces a bijection between the closed $G$-orbits in $X$ and the points of $Y$. 
\end{definition}

\begin{lemma}Every good quotient is a categorical quotient.
\end{lemma}

\begin{proof}\cite[Proposition 6.1.7]{pot:97}
\end{proof}

In our point of view quotient always means categorical quotient and is denoted by $X//G$. Being a good (or geometric) quotient is an additional - and very desireable - property. 

\begin{definition}A morphism of algebraic groups $\chi: G \longrightarrow k^\ast$ is called a character of $G$. We denote the group of characters of $G$ with $\Char(G)$.
\end{definition}

\begin{example}The character group of $\GL_n$ is freely generated by the character $\det : \GL_n \longrightarrow k^\ast$.
\end{example}

Let $\V$ be a finite dimensional $k$-vector space, and let $\rho: G \longrightarrow \GL(\V)$ be a linear representation of $G$. Endow $\V$ with the induced $G$-action. Assume that $\Delta:= \ker \rho \subset G$ is irreducible. In our case, the group $G$ will always be a finite product of general linear groups $G=\GL_{n_1} \times \ldots \times \GL_{n_s}$. Recall that $\GL_n$ is reductive, as is any finite product of reductive groups.

\begin{definition}Let $\chi$ be a character of $G$. A $\chi$-invariant function of $\V$ is a regular function
\[ f: \V \longrightarrow k, \]
such that $f(gv)= \chi(g) f(v)$ for all points $v \in \V$.

We denote the $k$-vector space of  $\chi$-invariant functions of $\V$ by $k[\V]^G_{\chi}$ and the $k$-algebra of invariant polynomials by $k[\V]^G$.
\end{definition} 

\begin{definition}Let $\chi \in \Char(G)$ be a character. A point $v \in \V$ is called
\begin{enumerate}
\item $\chi$-semistable, if and only if a $\chi^n$-invariant function $f \in k[\V]_{\chi^n}^G$  exists for some $n \in \N_{>0}$, such that $f(v) \neq 0$;
\item $\chi$-stable, if there is a $\chi^n$-invariant $f \in k[\V]_{\chi^n}^G$ with $n \geq 1$, $f(v) \neq 0$, such that the $G$-action on $\{ v \in \V \vbar f(v) \neq 0 \}$ is closed and $\dim G \cdot v = \dim G/\Delta$.
\end{enumerate}
\end{definition}

\begin{remark}In general $\chi$-(semi)stable points need not exist. Also it is non trivial in general to determine when there are semistable points that are not stable.
\end{remark}

We denote by $\V^{ss}(\chi)$ and $\V^{s}(\chi)$  the respective subsets of $\chi$-semistable and $\chi$-stable points. Their importance  lies in the fact that we can construct the quotient space
\begin{equation} \pi: \V^{ss}(\chi) \longrightarrow \V^{ss}(\chi) // G = \Proj \oplus_{n \geq 0} k[\V]^G_{\chi^n}. \end{equation}
It has many nice properties (see \cite{mfk:94}):
\begin{enumerate}
\item it is a good quotient;
\item the quotient space $\Proj \oplus_{n \geq 0} k[\V]^G_{\chi^n}$ is projective over $\Spec k[\V]^G$;
\item there is an open subset $U \subset \Proj \oplus_{n \geq 0} k[\V]^{\chi^n}$, such that $\pi^{-1}(U) = \V^{s}(\chi)$ and that the restriction
\begin{equation} \pi|_{\V^{s}(\chi)}: \V^{s}(\chi) \longrightarrow \V^{s}(\chi)//G:= U \end{equation}
is a geometric quotient.
\end{enumerate}

It follows from this description that the quotient $\V^{ss}(\chi) // G$ is projective over $k$ if and only if the ring of invariants consists only of the constant functions, i.e. if $k[\V]^G=k$.

\begin{remark}We used the terminology $\Spec$ and $\Proj$ in order to make a precise statement for the reader familiar with these notions. The functor $\Spec$ associates with every ring $A$ a geometric object, the affine scheme $\Spec A$. Its points are the prime ideals in $A$. If $A=k[X_1, \ldots, X_n] / \sqrt{ \langle f_1, \ldots, f_r \rangle}$ for some polynomials $f_1, \ldots, f_r \in k[X_1, \ldots, X_n]$, then $\Spec A$ can be identified with the zero-locus $\{ x \in k^n \vbar f_i(x) =0 \; \forall i=1, \ldots, r \}$.

The functor $\Proj$ associates with every graded ring $B= \oplus_{d \in \N} B_d$ a geometric object, the projective scheme $\Proj B$. Its points are the homogeneous prime ideals $\p \subset B$ that do not contain $\oplus_{d>0} B_d$. If $B = k[X_0, \ldots, X_n] / \sqrt{ \langle g_1 \ldots, g_r \rangle}$ for homogeneous polynomials $g_1, \ldots, g_r \in k[X_0, \ldots, X_n]$, then $\Proj B$ can be identified with the zero-locus $\{ x \in \proj{n} \vbar g_i (x) = 0 \; \forall i=1, \ldots, r \}$.
\end{remark}

\begin{example}Consider again the action of $\C^\ast$ on $\C^{n+1}$ by multiplication. Let $d \in \Z$ be an integer. The corresponding character is $\chi_d: \C^\ast \longrightarrow \C^\ast, \; z \mapsto z^d$. If $d<0$, then there are no $\chi_d$-invariant polynomial functions on $\C^{n+1}$. If $d\geq 0$, then the $\chi_d$-invariant functions are the homogeneous polynomials $f \in \C[X_0, \ldots, X_n]$ of degree $d$ that we denote by $\C[X_0, \ldots, X_n]_d$.

It follows that for $d<0$, the set of $\chi_d$-semistable points is empty. For $d=0$, every point $v \in \C^{n+1}$ is $\chi_d$-semistable, but none is $\chi_d$-stable. For $d>0$, a point $v \in \C^{n+1}$ is $\chi_d$-semistable and $\chi_d$-stable if and only if $v \neq 0$, so that $\left(\C^{n+1}\right)^s(\chi_d) = \left(\C^{n+1}\right)^{ss}(\chi_d) = \C^{n+1}- \{0\}$. 

We have already seen that $\C^{n+1}//\C^\ast = \{ \ast \}$. Now $\C^{n+1}-\{0\} = \left(\C^{n+1}\right)^{s}(\chi_1)$. Since $\chi_1^d = \chi_d$ and $k[\C^{n+1}]^{\C^\ast}_{\chi_d}= \C[X_0, \ldots, X_n]_{d}$ for $d \geq 0$, we conclude that $\C^{n+1}-\{0\} // \C^\ast = \Proj \oplus_{d \geq 0} \C[X_0, \ldots, X_n]_d = \proj{n}$ and that it is indeed a geometric quotient.

The main point is that in the whole space $\C^{n+1}$  there is only one closed orbit, namely the fixed point $\{0\}$ of the $\C^\ast$-action. If we remove it, the induced action of $\C^\ast$ on $\C^{n+1}-\{0\}$ is closed, i.e. all $\C^\ast$-orbits are closed in $\C^{n+1} - \{ 0 \}$ and we obtain an interesting geometric quotient. This explains the importance of choosing the right subsets to construct quotients in algebraic geometry. Geometric invariant theory provides us with a tool to systematically identify these subsets. 
\end{example}

\begin{theorem}\label{luna}Let $G=\GL_{v_1} \times \ldots \GL_{v_s}$ for some $v_1, \ldots, v_s \in \N$, and let $\chi \in \Char(G)$ be a character. If all stabilizers of the $G$-action on $\V^{s}(\chi)$ are trivial, then the quotient map
\[ \V^s(\chi) \longrightarrow \V^s(\chi)//G \]
is a principal $G$-bundle. 
\end{theorem}

\begin{proof}By Luna's Slice theorem, the quotient map is a principal bundle in the \'etale topology. See \cite[Corollary on page 199]{mfk:94}. Serre introduced in \cite{serre:58} the notion of special linear algebraic groups as being those for which all principal bundles in the \'etale topology are already locally trivial in the Zariski topology. He also proved that finite products of general linear groups are special. 
\end{proof}

\begin{definition}A quiver is a $4$-tuple $Q=(Q_0, Q_1, s,t)$ consisting of two finite sets $Q_0$ (the vertices) and $Q_1$ (the edges), and of two maps $s,t: Q_1 \longrightarrow Q_0$ which assign to an edge $\alpha \in Q_1$ its source $s(\alpha)$ and its tail $t(\alpha)$, respectively.
\end{definition}

\begin{definition}Let $Q=(Q_0, Q_1, s,t)$ be a quiver and fix $v \in \N^{Q_0}$. The affine space
\[ \Rep_Q^v:= \oplus_{\alpha \in Q_1} \Hom_k \bigl( k^{v_{s(\alpha)}}, k^{v_{t(\alpha )}} \bigr)\]
is called the space of representations of $Q$ of fixed dimension $v$.
\end{definition}

\begin{definition}Let $R=(r_\alpha)_{\alpha \in Q_1}$ be a representation of a quiver $Q$ of dimension $v \in \N^{Q_0}$. A subrepresentation of $R$ is a family of vector subspaces $S=(S_i)_{i \in Q_0}$, $S_i \subset k^{v_i}$ such that
\[ r_{\alpha}(S_{s(\alpha)}) \subset S_{t(\alpha)} \text{ for all } \alpha \in Q_1. \]
The dimension of $S$ is the vector $\dim S:= ( \dim S_i)_{i \in Q_0} \in \N^{Q_0}$.
The subrepresentation $S$ of $R$ is called proper if $0 \neq \dim S \neq \dim R$, i.e. if $S_i$ has positive dimension for some $i \in Q_0$ and if $S_j \subsetneq k^{v_j}$ is a proper subspace for some $j \in Q_0$.
\end{definition}

The reductive linear algebraic group $\GL_v:= \prod_{i \in Q_0}\GL_{v_i}$ acts naturally on $\Rep_Q^v$ by change of basis on all vertices, i.e.
\begin{equation} (g_i)_{i \in Q_0} \cdot (r_\alpha)_{\alpha \in Q_1}:= \Bigl( g_{t(\alpha)} r_\alpha g_{s(\alpha)}^{-1} \Bigr)_{\alpha \in Q_1}, \end{equation}
where $(r_\alpha)_{\alpha \in Q_1} \in \Rep_Q^{v}$, $(g_i)_{i \in Q_0} \in \GL_v$.

Fix a subset $M \subset Q_0$ of marked vertices. It defines a subgroup 
\begin{equation} \GL_{v,M}:= \prod_{i \in M} \GL_{v_i} \subset \GL_v. \end{equation}
We endow $\Rep_Q^v$ with the induced action. 

Every character $\chi \in \Char \bigl( \GL_{v,M} \bigr)$ is of the form 
\begin{equation} \chi: \GL_{v,M} \longrightarrow k^\ast, \; (g_i)_{i \in M} \mapsto \prod_{i \in M} \det(g_i)^{\chi_i} \end{equation}
for a family of integers $(\chi_i)_{i \in M} \in \Z^M$. Henceforth we indentify the character $\chi$ with the corresponding family of integers $(\chi_i)_{i \in M}$ and therefore the Character group $\Char \bigl( \GL_{v,M} \bigr)$ with $\Z^M$.

We usually indicate the data $(Q,v,M)$ by a diagram of the following form:
\begin{equation}  \xymatrix{ \overset{n}{\bullet} \ar@/^/[d] \ar@/_/[d] \ar@/^/[r] \ar[r] \ar@/_/[r]&   
                 \circ \, \scriptstyle{1} \ar@/^/[dl]  \ar@/_/[dl] \\
            \underset{n}{\bullet} &}.\end{equation}
The numbers denote the dimension we assign to the corresponding vertex, a filled dot indicates that this vertex is marked, i.e. belongs to $M$, an unfilled dot indicates that the vertex is unmarked.

King has found a helpful criterion for identifying the sets of semistable and stable representations of a given quiver which we shortly recall. He described it in the situation $M=Q_0$, i.e. where the full group $\GL_v$ acts on the representation space $\Rep_Q^v$. We are specifically interested in the case $M \subsetneq Q_0$, so we need to adjust it to our more general situation. This will be done subsequently.

Let $R \in \Rep_Q^v$ be a representation of $Q$ and fix a character $\chi \in \Char \bigl(\GL_{v,M}\bigr)$.
For any subrepresentation $S$ of $R$ we put
\begin{equation} \langle \chi, S \rangle_{_M}:= \sum_{i \in M} (\dim S)_i \cdot \chi_i \in \Z. \end{equation}

\begin{theorem}[King]\label{quiver_stability}Let $Q$ be a quiver, $v \in \N^{Q_0}$ a dimension vector, and $\chi \in \Char(\GL_v)$ a character. A point $R \in \Rep_Q^v$ is 
\begin{enumerate}
\item $\chi$-semistable, if and only if $\langle \chi, R \rangle_{_{Q_0}}=0$ and $\langle  \chi, S \rangle_{_{Q_0}} \geq 0$ for all proper subrepresentations $S$ of $R$;
\item $\chi$-stable, if and only if $\langle \chi, R \rangle_{_{Q_0}}=0$ and $\langle \chi, S \rangle_{_{Q_0}} > 0$ for all proper subrepresentations $S$ of $R$.
\end{enumerate}
\end{theorem}

\begin{proof}
\cite[Proposition 3.1]{king:94}.
\end{proof}

An (oriented) path of a quiver is a finite sequence of edges $\alpha_1, \ldots, \alpha_s$, s. th. $t(\alpha_i)=s(\alpha_{i+1})$ for all $i=1, \ldots, s-1$. A path is called an (oriented) cycle if additionally $s(\alpha_1) = t(\alpha_s)$. A cycle $(\alpha_1, \ldots, \alpha_s)$ defines a morphism $\Rep_{Q}^v \longrightarrow \End(k^{v_{s(\alpha_1)}}), \; R=(r_{\alpha})_{\alpha \in Q_1} \mapsto r_{\alpha_s} \circ \ldots \circ r_{\alpha_1}$. Composing this map with the trace function, we obtain a regular function on $\Rep_{Q}^v$. 

\begin{theorem}[Le Bruyn, Procesi]\label{procesi} The ring of polynomial invariants for the action of $\GL_{v}$ on $\Rep_Q^v$ is generated by traces of oriented cycles in the quiver $Q$ of length at most $N^2$, where $N= \sum v_i$.
\end{theorem}
\begin{proof} 
\cite[Theorem 1]{lebruyn:90}.
\end{proof}

If the quiver contains no oriented cycles, then $k[\V]^G = k$ and hence for any character $\chi \in \Char\bigl(\GL_v\bigr)$ the quotient $\V^{ss}(\chi)//\GL_v$ is projective over $k$. This special case has already been obtained by King (\cite[Proposition 4.3]{king:94}). If the quiver contains a cycle, then there is a non constant, invariant function, and hence $k \subsetneq k[\V]^G$, which in turn implies that the quotient is non projective. 

\begin{example}Let $(Q,v,M)$ be as specified by the diagram 
\begin{equation} \xymatrix{ \bullet \, \scriptstyle{n} \ar@(ul,dl)[]_{A} } \end{equation}
and assume $n>0$. Then $\Rep_Q^v=k^{n \times n}$ and $\GL_v = \GL_n$ acts by conjugation $(g, A) \mapsto gAg^{-1}$. Let $\chi \in \Z = \Char(\GL_n)$ be a character. 

By Theorem \ref{quiver_stability}, a point $R \in \Rep_Q^v$ is $\chi$-semistable if $\chi n =0$ and if for all $U \subset k^n$ with $R(U) \subset U$, we have $\chi \dim U \geq 0$. Hence for $\chi \neq 0$, there are no $\chi$-semistable points. But every point is $0$-semistable. However, $R$ is $0$-stable only if there is no proper $R$-invariant linear subspace $U \subset k^n$. But when $n>1$, then every matrix $R$ has such subspaces (choose any eigenvector $w \in k^n$ and put $U:=\langle w \rangle$).

Now we consider the ring $k[\Rep_Q^v]^{\GL_n}$ of $\GL_n$-invariant functions on $k^{n \times n}$. Taking traces of oriented cycles in the quiver $Q$ means considering the morphisms $t_l : k^{n \times n} \longrightarrow k, \, A \mapsto tr A^l$. 
Theorem \ref{procesi} tells us that 
\[ k[\Rep_Q^v]^{\GL_n} = k\left[t_1, \ldots, t_{n^2}\right].\]
In this case it suffices of course to take only the first $n$ functions $t_1, \ldots, t_n$.
\end{example}

We will need versions of both theorems for the case where $M \subsetneq Q_0$ is a proper subset. They can be obtained by applying the two reduction steps described below. Let us stress again that this reduction to the special situation is not new. Crawley-Boevey applied this method in \cite{boevey:2001} to the study of framed quiver moduli, initiated by Nakajima in \cite{nakajima:96}. In \cite{halic:04}, Halic and Stupariu use this reduction to  generalize Theorem \ref{procesi} to our situation (and hence they obtained Corollary \ref{general_procesi}). It is also implicitly used in \cite{geiss:06}[Appendix by Le Bruyn and Reineke]. Since we don't know of a reference where this reduction is presented in the generality we need, we include a description.

\subsubsection*{Step One}

Let $Q$ be a quiver,  $v \in \N^{Q_0}$ a dimension vector, and $M \subset Q_0$ a subset. We can assume w.l.o.g. that no edge connects two unmarked vertices, since they are unaffected by the group action.

\noindent Construct a new quiver $\tilde{Q}$ and a new dimension vector $\tilde{v} \in \N^{\tilde{Q}_0}$ as follows:

\begin{itemize}
\item collapse all unmarked vertices to a single one, denoted by $\infty$;
\item replace all edges $\alpha \in Q_1$ connecting a marked vertex $i \in M$ and an unmarked vertex $j \notin M$, by $v_j$ edges, connecting $i$ and $\infty$  (keep the orientation);
\item put $\tilde{v}_i:= v_i$ for $i \in M$ and $\tilde{v}_{\infty}:=1$;
\end{itemize} 

\begin{equation} \begin{array}{llll} 
\xymatrix{ \overset{n}{\bullet} \ar@/^/[d] \ar@/_/[d] \ar[r]& \overset{3}{\circ}  \\
            \underset{n}{\bullet} & \underset{2}{\circ} \ar[l]}    &
\begin{matrix} \\ \\ \\\mapsto \\ \end{matrix} &  & 

 \xymatrix{ \overset{n}{\bullet} \ar@/^/[d] \ar@/_/[d] \ar@/^/[r] \ar[r] \ar@/_/[r]&   
                 \circ \, \scriptstyle{1} \ar@/^/[dl]  \ar@/_/[dl] \\
            \underset{n}{\bullet} &} 
\end{array} \end{equation}

\begin{lemma}There  is a natural  $\GL_{v,M}$-equivariant isomorphism of varieties \[\Phi: \Rep_Q^v \iso \Rep_{\tilde{Q}}^{\tilde{v}}.\]
\end{lemma}

\begin{proof}
Arrows connecting marked vertices remain unchanged. So we have to look only at arrows that link marked with unmarked vertices. If the original quiver consists of an arrow from a marked vertex to an unmarked one with dimension vectors $n$ and $m$:
$\xymatrix{ \overset{n}{\bullet} \ar[r]& \overset{m}{\circ} }$,
then $\Rep_Q^v = \Hom(k^n, k^m)$ and $\Rep_{\tilde{Q}}^{\tilde{v}} = \Hom(k^n,k)^m$. In both cases the group acting on the representation space is $\GL_n$. The isomorphism $\Phi$ then maps a $m \times n$ matrix $A \in \Hom(k^n, k^m)$ to its $m$ rows $A_i \in \Hom(k^n, k)$. This defines an equivariant map, since
\begin{equation} \left( \begin{smallmatrix} A_1 \\ \vdots \\ A_m \end{smallmatrix} \right) g^{-1} = \left( \begin{smallmatrix} A_1 g^{-1} \\ \vdots \\ A_m g^{-1} \end{smallmatrix} \right) \text{ for any } g \in \GL_n. \end{equation}
Proceed similarly with arrows $\xymatrix{ \overset{n}{\bullet} & \overset{m}{\circ} \ar[l] }$ going from an unmarked arrow to a marked one.
\end{proof}

\subsubsection*{Step Two}

Let $Q$ be a quiver, $v \in \N^{Q_0}$ a dimension vector and $\infty \in Q_0$ a distinguished vertex with $v_{\infty}=1$. Put $M:= Q_0 \wo \{ \infty \}$ and choose a character $\chi \in \Char(\GL_{v,M})$. Replace $M$ by the full set $Q_0$ and observe that the set of (semi)stable points and the respective quotients do not change, if the character is suitably extended to a character of $\GL_v$. This is indicated in the following diagram.

\begin{equation} \begin{array}{llll} \xymatrix{ \overset{v_1}{\bullet} \ar@/^/[d] \ar@/_/[d] &  \\
            \underset{v_2}{\bullet} & \underset{1}{\circ} \ar[l]} 
& \begin{matrix} \\ \\ \\\mapsto \\ \end{matrix} &  & 

\xymatrix{ \overset{v_1}{\bullet} \ar@/^/[d] \ar@/_/[d] &  \\
            \underset{v_2}{\bullet} & \underset{1}{\bullet} \ar[l]}  \\ & & &\\

\chi=(\chi_1, \chi_2) & \mapsto & & \tilde{\chi}=(\chi_1,\chi_2, -\chi_1 v_1 - \chi_2 v_2) 
\end{array}
\end{equation}

To be more precise, define the character $\tilde{\chi} \in \Char(\GL_v)$ by \begin{equation} \tilde{\chi}_{\infty}:= - \sum_{i \in M} \chi_i v_i, \; \; \tilde{\chi}_{i}:= \chi_i \; \text{f\"ur } i \in M.\end{equation}

\begin{lemma}\label{step2}
The $\chi$-(semi)stable points in $\Rep_{Q}^v$ (with respect to the action of $\GL_{v,M}$) are the  $\tilde{\chi}$-(semi)stable points in $\Rep_{Q}^v$ (with respect to the action of $\GL_{v}$), and the corresponding quotients agree.
\end{lemma}

\begin{proof}The crucial remark is that the action of $\GL_v$ on $\Rep_Q^v$ is induced by the action of $\GL_{v, M}$ and the morphism $\alpha: \GL_v \longrightarrow \GL_{v,M}, \; \left( (g_m)_{m \in M}, z \right) \mapsto \left(z^{-1}g_m \right)_{m \in M}$. Furthermore, we have $\tilde{\chi} = \chi \circ \alpha$. Therefore the $\tilde{\chi}^k$-invariant functions are exactly the $\chi^k$-invariant functions. It follows immediately that the respective subsets of (semi)stable points coincide. 

Now let $U \subset \Rep_Q^v$ be any $\GL_v$- (and hence $\GL_{v,M}$)-invariant open subset. A morphism $U \longrightarrow Y$ into any variety $Y$ is $\GL_v$-invariant if and only if it is $\GL_{v,M}$-invariant. Therefore the quotient $U//\GL_v$ exists if and only if the quotient $U//\GL_{v,M}$ exists, in which case they agree.
\end{proof}

Putting both steps together, we reduce the general setup, given by a quiver $Q$, a dimension vector $v$, and a set of marked vertices $M \subset Q_0$ to the situation where the full group $\GL_{\tilde{v}}$ acts on the representation space $\Rep_{\tilde{Q}}^{\tilde{v}}$. \\

Now we are able to formulate versions of the theorems cited above for the general situation.

\begin{corollary}\label{general_quiver_stability}Let $Q$ be a quiver, $v \in \N^{Q_0}$ a dimension vector, and $M \subset Q_0$ a subset of marked vertices. Fix a character $\chi \in \Char\bigl(\GL_{v,M}\bigr)$. 

 A representation $R=(r_{\alpha})_{\alpha \in Q_1} \in \Rep_Q^v$ is $\chi$-semistable if and only if 
\begin{enumerate}
\item all subrepresentations $S=(U, \psi)$ of $R$ with $U_m = \{ 0 \}$ for all $m \notin M$ satisfy $\langle \chi, S \rangle_{_M} \geq 0$;
\item all subrepresentations $S=(U, \psi)$ of $R$ with $U_m = k^{v_m}$ for all $m \notin M$ satisfy $\langle \chi, S \rangle_{_M} \geq \langle \chi, R \rangle_{_M}$.
\end{enumerate}
 The representation is $\chi$-stable, if and only if the inequalities are strict whenever these subrepresentations are proper.
\end{corollary}

\begin{proof}Let $(\tilde{Q}, \tilde{v}, \tilde{M}, \tilde{\chi})$ be the data obtained by applying the reduction procedure. We have a natural isomorphism $\Phi: \Rep_{v,Q} \iso \Rep_{\tilde{v}, \tilde{Q}}$, and we have to check whether the representation $\Phi(R)$ is $\tilde{\chi}$-(semi)stable with respect to the $\GL_{\tilde{v}}$-action. Recall that all unmarked vertices in $Q$ are collapsed to a single one in $\tilde{Q}$, denoted by $\infty$,  of dimension $1$. By Theorem \ref{quiver_stability} we have to consider subrepresentations  of $\Phi(R)$. At the vertex $\infty$, it is either the null space $\{ 0 \}$ or the full space $k$. Indeed, the subrepresentations of $S$ of $R$ that assign either the null space or the full space $k^{v_m}$ to all the unmarked vertices $m \notin M$ correspond bijectively to the subrepresentations $\tilde{S}$ of $\Phi(R)$. Furthermore 
\begin{equation}\begin{split}
\langle \tilde{\chi}, \tilde{S} \rangle_{\tilde{Q}_0} &= \sum_{m \in M} \chi_m \dim \tilde{S}_m + \tilde{\chi}_{\infty} \dim \tilde{S}_{\infty} \\
&= \langle \chi, S \rangle_{M} - \langle \chi, R \rangle_M \dim \tilde{S}_{\infty}.
\end{split}\end{equation}
In particular $\langle \tilde{\chi}, \Phi(R) \rangle_{\tilde{Q}_0} = 0$. Now the statement is an immediate consequence of Theorem \ref{quiver_stability}.
\end{proof}

To formulate Theorem \ref{procesi} in the more general setup, call a cycle marked if it starts (and ends) at a marked vertex. If we are given a path connecting two unmarked vertices in $(Q,M)$ it  defines a map $\Rep_Q^v \longrightarrow \Hom(k^{v_{s(\alpha_1)}}, k^{v_{t(\alpha_s)}})$. Composing with the natural coordinate functions on the latter space, it induces $v_{s(\alpha_1)} v_{t(\alpha_s)}$ natural regular functions on $\Rep_Q^v$. The paths connecting unmarked vertices correspond to $v_{s(\alpha_1)} v_{t(\alpha_s)}$ cycles starting (and ending) at $\infty$ in the quiver $\tilde{Q}$. Under the isomorphism $\Rep_{Q}^v \cong \Rep_{\tilde{Q}}^{\tilde{v}}$ the traces of the latter cycles correspond exactly to the coordinate functions of the former paths. Hence we obtain the following corollary (this result has been obtained by Halic and Stupariu in \cite{halic:04}):

\begin{corollary}[Halic, Stupariu]\label{general_procesi}
Let $Q$ be a quiver, $v \in \N^{Q_0}$ a dimension vector, and $M \subset Q_0$ a subset of marked vertices. Fix a character $\chi \in \Char\bigl(\GL_{v,M}\bigr)$. 

The ring of invariants $k[\Rep_Q^v]^{\GL_{v,M}}$  is generated by traces of marked cycles and by coordinate functions of paths with unmarked source and tail. 
 In particular, the quotient space of $\chi$-semistable representations by $\GL_{v,M}$ is projective if and only if there are no cycles and no paths with unmarked source and tail. 
\end{corollary}

For an example consider the space of linear dynamical systems: Fix $(n,m,p) \in \N_{>0} \times \N^2$. A linear dynamical system of type $(n,m,p)$ is a triple $\Sigma = (A,B,C,D)$ of matrices $A \in k^{n \times n}, \, B \in k^{n \times m}$,  $C \in k^{p \times n}$, and $D \in k^{p \times m}$. 
The affine space of all linear dynamical systems of type $(n,m,p)$, which we will denote by $\tSigma_{n,m,p}$, is the representation space $\Rep^{(n,m,p)}_Q$ of the quiver $Q$ described by the following diagram
\begin{equation} \xymatrix{ \overset{n}{\bullet \ar@(ul,dl)[]_{A}} \ar[d]_{C} & \circ \, \scriptstyle {m} \ar[l]_{B} \ar[dl]^{D} \\ \underset{p}{\circ} & } \end{equation}
In control theory, linear dynamical systems are studied up to change of basis in the state space $k^n$. This is the action of $\GL_n \subset \GL_{n,m,p}$ on the representation space we obtain by marking only the inner vertex. \\

We want to identify the respective sets of $\chi$-(semi)stable points using Corollary \ref{general_quiver_stability}. This is surprisingly simple and thus shows the power of this machinery. Byrnes and Hurt in \cite{bh:79} first used geometric invariant theory to construct and study these quotients. Later Tannenbaum \cite{tannenbaum:91} did the same. See also \cite[Chapter 7]{moore:94}. There is a recent article \cite[in particular Proposition 4.1 and Lemma 5.4.1]{geiss:06}, where a different proof of the following proposition involving the controllability matrix and the Hilbert-Mumford criterion of stability is given and the graded ring $\oplus_{ k \geq 0} k [ \tSigma_{n,m,p}]_{\det^k}^{\GL_n}$ is described. See \cite{corfra:03} for the notions of controllability and observability of linear dynamical systems.

\begin{proposition}[Byrnes, Hurt]Let $\chi \in \Z \cong \Char(\GL_n)$ be a character. 
\begin{enumerate}
\item When $\chi > 0$, then a point $\Sigma \in \tSigma_{n,m,p}$ is $\chi$-stable if and only if it is $\chi$-semistable if and only if it is controllable.
\item When $\chi=0$, then every system $\Sigma \in \tSigma_{n,m,p}$ is $\chi$-semistable. It is $\chi$-stable if and only if it is controllable and observable.
\item When $\chi<0$, then a system $\Sigma \in \tSigma_{n,m,p}$ is $\chi$-stable if and only if it is $\chi$-semistable if and only if it is observable.
\end{enumerate}
\end{proposition} 

\begin{proof}
Let $\Sigma=(A,B,C,D) \in \Rep_Q^v$ be a representation. A subrepresentation $S$ of $R$ is given by a triple of subspaces $U \subset k^n$, $V \subset k^p$, and $W \subset k^m$, such that $A(U) \subset U$, $B(W) \subset U$, $D(W) \subset V$, and $C(U) \subset V$. We have to consider the cases a) $V=k^p$ and $W=k^m$ and b) $V=W=\{0\}$. The first situation corresponds to finding an $A$-invariant subspace $U \subset k^n$, such that $\im B \subset U$. The second situation means finding an $A$-invariant subspace $U \subset k^n$, such that $\ker C \supset U$.
 Since $\langle \chi, S \rangle_{_M} = \chi \dim U $ and $\langle \chi, R \rangle_{_M} = \chi n$, we derive the following two conditions for $\chi$-semistability from Corollary \ref{general_quiver_stability}:
\begin{enumerate}
\item $\bigl( U \subsetneq k^n, \, A(U) \subset U, \, \im B \subset U \bigr)\implies \chi \dim U \geq \chi n$;
\item $\bigl( U \neq 0, \, A(U) \subset U, \, \ker C \supset U \bigr) \implies \chi \dim U \geq 0$.
\end{enumerate} 
Replace the inequalities by strict ones to obtain the stability criteria.
When $\chi=0$, then all representations $\Sigma$ are $\chi$-semistable. In the other cases all semistable points are stable and the stability conditions translate to
\begin{enumerate}
\item $\nexists \, U \subsetneq k^n$ with $A(U) \subset U$ and $\im B \subset U$, when $\chi>0$;
\item $\nexists \, 0 \neq U \subset k^n$ with $A(U) \subset U$ and $\ker C \supset U$, when $\chi<0$.
\end{enumerate}

As a consequence, when $\chi>0$, a representation  $\Sigma=(A,B,C,D)$ is $\chi$-(semi)stable if and only if 
\begin{equation} \textstyle\sum_{j \geq 0} \im A^jB = k^n, \text{ i.e. if and only if } \Sigma \text{ is controllable.} \end{equation}
To prove this, let $\Sigma$ be a $\chi$-stable system. Put $U:= \sum_{j \geq 0} A^jB$. Then $A(U) \subset U$ and $\im B \subset U$. The $\chi$-stability now implies $U = k^n$. Hence $\Sigma$ is controllable.
On the other hand, choose $\Sigma$ controllable. Let $U\subset k^n$ be any subspace with $A(U) \subset U$ and $\im B \subset U$. Then necessarily $\sum_{j \geq 0} \im A^jB \subset U$, whence $U = k^n$ by the controllability.
Therefore the $\chi$-(semi)stable points for any $\chi>0$ are exactly the controllable linear dynamical systems.

 Analogously, when $\chi<0$, a representation $\Sigma=(A,B,C,D)$ is $\chi$-(semi)stable if and only if
\begin{equation} \cap_{j \geq 0} \ker CA^j = \{ 0 \}. \end{equation}
Therefore the $\chi$-(semi)stable points for any $\chi<0$ are exactly the observable linear dynamical systems.

For $\chi=0$, we see that the $\chi$-stable points are those linear dynamical systems that are both controllable and observable.
\end{proof}

The respective sets $\tSigma_{n,m,p}^c$ and $\tSigma_{n,m,p}^o$ of controllable and observable systems are thus the stable loci of the action of $\GL_{n}$ with respect to the characters $\pm 1$. We write
\begin{equation} \Sigma_{n,m,p}^c := \tSigma_{n,m,p}^c//\GL_{n} \text { and } \Sigma_{n,m,p}^o:= \tSigma_{n,m,p}^o \end{equation}
for the respective quotients. 

\begin{corollary}The moduli spaces $\Sigma_{n,m,p}^c$ and $\Sigma_{n,m,p}^o$ of controllable and observable linear dynamical systems are smooth irreducible algebraic varieties of dimension $n(m+p)+pm$. They are non-projective and the quotient maps
\[ \tSigma_{n,m,p}^c \longrightarrow \Sigma_{n,m,p}^c \text{ and } \tSigma_{n,m,p}^o \longrightarrow \Sigma_{n,m,p}^o \]
are principal $\GL_n$-bundles.
\end{corollary}

\begin{proof}We apply Theorem \ref{luna}. All we have to prove is that the stabilizer of $\GL_n$ of a controllable (observable) system $\Sigma$ is trivial. Let $\Sigma=(A,B,C,D)$ be a controllable system. Then its controllability matrix $R(\Sigma)=[B, AB, \ldots, A^{n-1}B]$ has rank $n$. Let $g \in \GL_n$ be an element of the stabilizer of $\Sigma$, i.e. $g\Sigma = \Sigma$. Then $R(\Sigma) = R(g\Sigma) = gR(\Sigma)$. Therefore $g=\id_n$. Proceed analogously to prove the statement for observable systems. 
\end{proof}


\section{The Lomadze compactification}
 
\subsection{Lomadze systems}

Fix non-negative integers $n,m,p \in \N$ with $n>0$. The Lomadze compactification arises from generalizing the equations (\ref{lds}) to
\begin{equation}
Kw_{t+1} +Lw_t + M\xi_t=0,
\end{equation}
with matrices $K,L \in k^{(n+p) \times n}$ and $M \in k^{(n+p) \times (p+m)}$. Therefore we make the following definition:

\begin{definition}A Lomadze system of type $(n,m,p)$ is a triple $(K,L,M)$ consisting of matrices $K,L \in k^{(n+p) \times n}$ and $M \in k^{(n+p) \times (p+m)}$.
\end{definition}

We denote the vector space of all Lomadze systems by $\tL_{n,m,p}$. The reductive linear algebraic group $\GL_{n, n+p} = \GL_n \times \GL_{n+p}$ acts on $\tL_{n,m,p}$ as follows:
\begin{equation} (g_0,g_1). (K,L,M) = (g_1Kg_0^{-1}, g_1Lg_0^{-1}, g_1M) \text{ for } (g_0,g_1) \in \GL_{n,n+p} \text{ and } (K,L,M) \in \tL_{n,m,p}. \end{equation}

In \cite{lomadze:90}, Lomadze introduced notions of controllability and observability for systems $S \in \tL_{n,m,p}$:  

\begin{definition}\label{lc}A system $S = (K,L,M) \in \tL_{n,m,p}$ is called controllable, if it satisfies  
\begin{enumerate}
\item $\rk [ s K + t L] =n$ for some $(s, t) \in k^2$; 
\item $\rk [ s K + t L, M ] = n+p$ for all $0 \neq (s, t) \in k^2$. 
\end{enumerate} 
The system $S$ is called observable, if and only if
\begin{enumerate}
\item $\rk [s K + t L] = n$ for all $ 0 \neq (s, t) \in k^2$;
\item $\rk [s K + t L, M ] = n+p$ for some $(s, t) \in k^2$.
\end{enumerate}
We write $\tL_{n,m,p}^c$  and $\tL_{n,m,p}^o$ for the respective sets of controllable and observable systems. 

A system $S=(K,L,M) \in \tL_{n,m,p}$ is called regular, if $\rk[K, M_p]= n+p$ where $M_p \in k^{(n+p) \times p}$ consists of the first $p$ columns of the matrix $M$. We denote the space of regular systems with $\tL_{n,m,p}^{\reg}$.
\end{definition}

\begin{remark}The subsets of controllable and observable Lomadze systems are open and invariant with respect to the $\GL_{n,n+p}$-action.
\end{remark}

There is a natural map $\Phi_L: \tSigma_{n,m,p} \longrightarrow \tL_{n,m,p}$ that maps a system $(A,B,C,D) \in \tSigma_{n,m,p}$ to the Lomadze system $(K,L,M)$ defined by
\begin{equation} K = \begin{pmatrix} \Id_n \\ 0 \end{pmatrix}, \, L= \begin{pmatrix} A \\ C \end{pmatrix}, \, M= \begin{pmatrix} 0 & B \\ \Id_p & D \end{pmatrix}. \end{equation}

Furthermore, there is a natural inclusion of groups $\varphi_L: \GL_n \longrightarrow \GL_{n,n+p}, \; g_0 \mapsto (g_0, g_0 \oplus id_p)$. 

\begin{lemma}\label{inclusion}The map $\Phi_L: \tSigma_{n,m,p} \longrightarrow \tL_{n,m,p}$ is a closed embedding and it is $\varphi_L$-equivariant, i.e. for all $\Sigma \in \tSigma_{n,m,p}$ and for all $g \in \GL_n$ 
\[ \Phi_L(g. \Sigma) = \varphi_L(g). \Phi_L(\Sigma). \]
Furthermore: 
\begin{enumerate}
\item Given a regular system $S \in \tL_{n,m,p}$, there exists $g \in \GL_{n,n+p}$, such that $g.S \in \im \Phi_L$;
\item Given a system $\Sigma \in \tSigma_{n,m,p}$ and $g \in \GL_{n,n+p}$ with $g \Phi_L(\Sigma) \in \im \Phi_L$, then $g = \varphi_L(g_0)$ for some $g_0 \in \GL_n$.
\item A system $\Sigma \in \tSigma_{n,m,p}$ is controllable (observable) if and only if $\Phi_L(\Sigma)$ is controllable (observable).
\end{enumerate}

\end{lemma}

\begin{proof}This follows from straightforward calculations. See \cite[Lemma 3.14]{diplomarbeit} for details.
\end{proof}

The following statement has been proved by Helmke for $p=0$ in \cite[Lemma 6.4]{helmke:92}. The case $p>0$ is implicit in \cite{HRS:97}, where a reachability matrix for Lomadze systems is defined. The proof given here follows closely these references and is included for the reader's convenience. 

\begin{proposition}\label{stabilizer}Let $S=(K,L,M) \in \tL_{n,m,p}$ be a controllable or observable Lomadze system. Then the stabilizer of the $\GL_{n,n+p}$-action at the point $S$ is trivial.
\end{proposition}

\begin{proof}First, let us introduce an $\GL_{n,n+p}$-equivariant automorphisms on the space of all Lomadze systems:  To an invertible matrix $\Omega= \left( \begin{smallmatrix} \alpha & \beta \\ \gamma & \delta \end{smallmatrix} \right) \in \GL_2$ and a matrix $h \in \GL_{p+m}$ , we associate the transformation
\[ T_{\Omega, h} : \tL_{n,m,p} \longrightarrow \tL_{n,m,p}, \; (K,L,M) \mapsto (\alpha K + \beta L, \gamma K + \delta L, Mh). \]

We now claim the following: Let $S \in \tL_{n,m,p}$ be a Lomadze system, such that there exists $(s,t) \in k^2$ verifying both $\rk [sK + tL] = n$ and $\rk [sK + tL,M]= n+p$. Then there exists $\Omega \in \GL_2$ and $h \in \GL_{m+p}$, such that $T_{\Omega, h}(S)$ is regular.
 
To prove this claim, take $(s,t)$ as in the assumption. Then there certainly exist numbers $u,v \in k$, such that $\Omega:=\left( \begin{smallmatrix} s & t \\ u & v \end{smallmatrix} \right)$ is invertible. Now we choose $h \in \GL_{p+m}$ to be a permutation matrix, such that $Mh$  consists of the same columns as $M$, but reordered in such a way that the first $p$ columns together with the columns of $[sK + tL]$ span the whole space $k^{n+p}$. Then $T_{\Omega, h}(S)$ is regular.

Notice that every controllable or observable system verifies the assumption of the preceeding claim. Since $T_{\Omega, h}$ is a $\GL_{n,n+p}$-equivariant isomorphism, a system $S$ has trivial stabilizer if and only if $T_{\Omega, h}(S)$ has trivial stabilizer. Therefore we have reduced the statement to the case of regular systems. So assume that $S=(K,L,M) \in \tL_{n,m,p}$ is a regular controllable or observable system. Then it is equivalent to a system $S'=\Phi_L(\Sigma)$ for some $\Sigma \in \tSigma_{n,m,p}$. Now let $g=(g_0,g_1) \in \GL_{n,n+p}$ lie in the stabilizer of $S'$. Since $g S' = S' \in \im \Phi_L$, we can apply Lemma \ref{inclusion} and conclude that $g=\varphi_L(g_0)$ for some $g \in \GL_n$. But using  $gS'=  \varphi_L(g_0) \Phi_L(\Sigma) = \Phi_L(g\Sigma)$, it follows from $gS' = S'$ that $g\Sigma = \Sigma$. The system $\Phi_L(\Sigma)$ is controllable (resp. observable) if and only if $\Sigma$ is controllable (resp. observable). But the stabilizer of a controllable or observable system $\Sigma \in \tSigma_{n,m,p}$ in $\GL_n$ is trivial. Therefore $g_0= \id_n$. Hence $S'$ has trivial stabilizer. But since the stabilizers of $S$ and of $S'$ are conjugate also $S'$ has trivial stabilizer. 
\end{proof}

\subsection{GIT quotients of Lomadze systems}

We would like to construct the quotients $\tL_{n,m,p}^c // \GL_{n,n+p}$ and $\tL_{n,m,p}^o // \GL_{n,n+p}$ using geometric invariant theory. In order to do this, we need to exhibit the respective subsets of controllable and observable Lomadze systems as the stable loci of some characters of the group $\GL_{n,n+p}$.

Let $Q_L$ be the following quiver:

\begin{equation} \xymatrix{ \stackrel{n}{\bullet} \ar@/^/[d]^L \ar@/_/[d]_K & \\ 
              \underset{n+p}{\bullet} & \underset{p+m}{\circ} \ar[l]_M } \end{equation}

Notice that the space of representations of dimension $(n, n+p, p+m)$ of this quiver is the space of Lomadze systems:  $\Rep_{Q_L}^{(n,n+p,p+m)} = \tL_{n,m,p}$. Since the group $\GL_{n,n+p}$ acts on $\tL_{n,m,p}$ by change of basis on the respective vertices, we are able to apply the formalism of the first section of this article. The first question we ask is: How many different stability notions are there? We had associated to every character $(\chi_0, \chi_1) \in \Z^2 \cong \Char(\GL_{n,n+p})$ a stability condition. We will give an elementary proof that there are only finitely many different stability conditions by grouping together characters that actually have coinciding (semi)stable loci. 

\begin{lemma}Let $(\chi_0,\chi_1) \in \Z^2 \cong \Char(\GL_{n, n+p})$ be a character. If  $(\chi_0,\chi_1) \in \Z^2$ satisfies one of the following three inequalities:
\[ \chi_0>0,\; \chi_1<0, \;n\chi_0 + (n+p)\chi_1<0,\]
 then there are no $(\chi_0,\chi_1)$-semistable systems in $\tL_{n,m,p}$. 
\end{lemma}

\begin{proof}
To test $(\chi_0,\chi_1)$-stability of a system $S=(K,L,M) \in \tL_{n,m,p}$, we need to consider vector subspaces $U \subset k^n$ and $V \subset k^{n+p}$. The criteria of Corollary \ref{general_quiver_stability} read as follows:
\begin{itemize}
\item if $K(U) + L(U) \subset V$, then $\chi_0 \dim U + \chi_1 \dim V \geq 0$;
\item if $K(U) + L(U) \subset V$ and $\im M \subset V$, then $\chi_0 \dim U + \chi_1 \dim V \geq \chi_0 n + \chi_1 (n+p)$.
\end{itemize} 

Put $U:=\{0\}$ and $V:=k^{n+p}$. Clearly $K(U) + L(U) \subset V$ and $\im M \subset V$. If $\chi_0 >0$, this contradicts the second inequality, if $\chi_1 < 0$, it contradicts the first one.
Now assume that $n \chi_0 + (n+p) \chi_1 < 0$. Put $U:=k^n$ and $V:=k^{n+p}$ to obtain a contradiction of the first inequality.
\end{proof}

Notice furthermore that a point $S=(K,L,M)$ is $(\chi_0, \chi_1)$-(semi)stable, if and only if it is $(l \chi_0, l \chi_1)$-(semi)stable for any $l \in \N_{>0}$. Therefore stability with respect to the character $(\chi_0, \chi_1)$ only depends on the fraction $- \frac{\chi_0}{\chi_1} \in [0, \infty]_{\Q}:= \{ q \in \Q \vbar q \geq 0 \} \cup \{ \infty \}$. 

Let $\chi \in [0, \infty]_{\Q}$. We will say that a point $S=(K,L,M)$ is $\chi$-(semi)stable if it is $(\chi_0, \chi_1)$-(semi)stable for some (and hence for every) character $(\chi_0,\chi_1) \in \Z_{\leq 0} \times \Z_{\geq 0}$ with $\chi = - \frac{\chi_0}{\chi_1}$.

Before proceeding, let us spell out the criteria of Corollary \ref{general_quiver_stability} in this situation: Let $\chi \in [0,\infty]_{\Q}$. A system $S=(K,L,M)$ is $\chi$-semistable, if and only if for all subspaces $U\subset k^n$ and $V \subset k^{n+p}$ the following two conditions hold:
\begin{enumerate}
\item if $\left( U \neq \{0 \} \text{ or } V \neq \{0 \} \right)$ and $K(U) + L(U) \subset V$, then $\chi \leq \frac{ \dim V } { \dim U}$.
\item if $\left( U \subsetneq k^n \text{ or } V \subsetneq k^{n+p} \right)$, $K(U) + L(U) \subset V$, and $\im M \subset V$, then $\chi \geq \frac{n+p - \dim V}{n - \dim U}$.
\end{enumerate}
Again, to obtain the stability criteria, replace the inequalities by strict ones. Let $\chi \in [0, \infty]_{\Q}$ and assume that the corresponding stable and semistable loci differ, i.e. that there exists a system that is $\chi$-semistable, but not $\chi$-stable. It follows from the criteria above that $\chi = \frac{v}{u}$ for some $ 0 \leq v \leq n+p$ and $0 \leq u \leq n$. Therefore we see that all parameters $\chi \in [0, \infty]_{\Q}$ admitting $\chi$-semistable but not $\chi$-stable points are contained in the finite set
\begin{equation} \Lambda:= \left\{ \frac{v}{u} \in [0, \infty]_{\Q} \vbar 0 \leq v \leq n+p, \; 0 \leq u \leq n \right\}.\end{equation}

\begin{definition}Let $\lambda \in \Lambda$. A system $S=(K,L,M)$ is called $\lambda$-lowerstable (respectively $\lambda$-upperstable), if for all subspaces $U \subset k^n$ and $V \subset k^{n+p}$ the following hold:
\begin{enumerate}
\item if $\left( \dim U >0 \text{ or } \dim V > 0 \right)$ and $K(U) +L(U) \subset V$, then 
 \[ \lambda < \frac{\dim V}{\dim U} \left( \text{respectively }\lambda \leq \frac{ \dim V}{\dim U} \right); \]
\item if $\left( U \subsetneq k^n \text{ or } V \subsetneq k^{n+p} \right)$, $K(U) + L(U) \subset V$, and $\im M \subset V$, then 
\[ \lambda \geq \frac{ n+p - \dim V}{n - \dim U} \left(\text{respectively } \lambda > \frac{ n+p - \dim V}{n - \dim U} \right). \]
\end{enumerate}
\end{definition}

Notice that the notion of $\lambda$-upper(lower)stability is a compromise between $\lambda$-semistability and $\lambda$-stability. A priori, this is not a GIT-stability condition. 

With every parameter $\lambda \in \Lambda$ we associate the following two intervals in $[0, \infty]_{\Q}$:
\begin{enumerate}
\item $\Delta_{\lambda}^{-}:= \left \{ \chi \in [0, \infty]_{\Q}  \vbar  \chi > \lambda \text{ and } \chi < \lambda' \text{ for all } \lambda<\lambda' \in \Lambda \right\}$.
\item $\Delta_{\lambda}^{+}:= \left \{ \chi \in [0, \infty]_{\Q}  \vbar  \chi < \lambda \text{ and } \chi > \lambda' \text{ for all } \lambda<\lambda' \in \Lambda \right\}$.
\end{enumerate}

If we write $\Lambda =\{ \lambda_1, \lambda_2, \ldots, \lambda_s\}$ and order the parameters such that $0 \leq \lambda_1 < \lambda_2 < \ldots < \lambda_s \leq \infty$ holds, then $\Delta_{\lambda_i}^- = \Delta_{\lambda_{i+1}}^+$ and these intervals are the connected components of $[0, \infty]_{\Q} - \Lambda$.

\begin{proposition}\label{h_chambers}Let $\lambda \in \Lambda$ and let $S=(K,L,M) \in \tL_{n,m,p}$. Let $\chi \in \Delta_\lambda^+$ (respectively in $\Delta_\lambda^-$). Then the following statements are equivalent:
\begin{enumerate}
\item $S$ is $\lambda$-upperstable (respectively $\lambda$-lowerstable);
\item $S$ is $\chi$-stable;
\item $S$ is $\chi$-semistable.
\end{enumerate}
\end{proposition}

\begin{proof}We prove the statement for $\chi \in \Delta_\lambda^+$. The proof for $\chi \in \Delta_\lambda^-$ is analogous. Fix $\chi \in \Delta_{\lambda}^+$.  Assume $S$ to be $\lambda$-upperstable. Let $U \subset k^n$ and $V \subset k^{n+p}$ be vector subspaces. Assume that $K(U) + L(U) \subset V$. If $\dim U + \dim V > 0$, then by assumption $\lambda < \frac{\dim V}{\dim U} \in \Lambda$. Therefore $\chi < \frac{\dim V}{\dim U}$. 

Now assume that $K(U) + L(U) \subset V$, $\im M \subset V$, and $\dim U + \dim V < 2n +p$. Then $\chi > \lambda \geq \frac{n+p - \dim V}{n - \dim U}$. Therefore $S$ is $\chi$-stable. 

Every $\chi$-stable system is $\chi$-semistable. So let us prove the last implication. Assume that $S$ is $\chi$-semistable. Again, given subspaces $U \subset k^n$ and $V \subset k^{n+p}$, we have to distinguish two cases: First, if $\dim U + \dim V > 0$ and $K(U) + L(U) \subset V$, then $\lambda < \chi  \leq \frac{\dim V}{\dim U}$. Secondly, if $\dim U + \dim V < 2n +p$ and $K(U) + L(U) \subset V \supset \im M$, then $\chi \geq \frac{n+p - \dim V}{n - \dim U} =: \lambda' \in \Lambda$. Therefore $\lambda \geq \lambda'$. This proves that $S$ is $\lambda$-upperstable.
\end{proof}
 
We see that the notion of $\chi$-stability is preserved when the parameters $\chi$ vary within an interval $\Delta_\lambda^+$: for every $\chi \in \Delta_\lambda^+$ a system $S$ is $\chi$-stable if and only if it is $\chi$-semistable if and only if it is $\lambda$-upperstable. Therefore the condition of being $\lambda$-upperstable is a GIT-stability condition. And there are only finitely many GIT-stability conditions.

\begin{corollary}Let $\lambda \in \Lambda$. A system $S \in \tL_{n,m,p}$ is $\lambda$-stable if and only if it is $\lambda$-upper- and $\lambda$-lowerstable.
\end{corollary}

\begin{proposition}\label{stability_generalization}Let $\Sigma \in \tSigma_{n,m,p}^c$ be a linear dynamical system and let  $\Phi_L: \tSigma_{n,m,p} \longrightarrow \tL_{n,m,p}$ be the closed embedding defined above.
\begin{enumerate}
\item The system $\Sigma \in \tSigma_{n,m,p}$ is $-1$-stable if and only if $\Phi_L(\Sigma) \in \tL_{n,m,p}$ is $1$-lowerstable.
\item The system $\Sigma \in \tSigma_{n,m,p}$ is $1$-stable if and only if $\Phi_L(\Sigma) \in \tL_{n,m,p}$ is $1$-upperstable.
\end{enumerate}
\end{proposition}

\begin{proof}
Put $S:= \Phi_L(\Sigma)$. Write $S=(K,L,M)$, assume $\Sigma$ to be $-1$-stable, and choose vector subspaces $U\subset K^n$ and $V \subset k^{n+p}$, such that $K(U) + L(U) \subset V$. From the equality $K = \left(\begin{smallmatrix} \Id_n \\ 0 \end{smallmatrix} \right)$ we know that $U \subset V$ (once we have identified $k^n$ with its image in $k^n \oplus k^p = k^{n+p}$). Now $U=V$ is only possible when $U$ is trivial, since otherwise the fact that $L(U)= \left(\begin{smallmatrix} A \\ C \end{smallmatrix} \right) (U) \subset V$ implies that $A(U) \subset U$ and $\ker C \supset U$. This is excluded by the observability of $\Sigma$. If $U \subsetneq V$, then $- \dim U + \dim V >0$. Now assume additionally that $\im M \subset V$. From $M= \left(\begin{smallmatrix} 0 & B \\ \Id_p & D \end{smallmatrix}\right)$ we conclude that $U \oplus k^p \subset V$ and thus that $\dim V \geq \dim U + p$. Therefore $-(n-\dim U) + (n+p - \dim V) = \dim U + p - \dim V \leq 0$. Hence the system $S$ is $1$-lowerstable.

Now suppose $S$ to be $1$-lowerstable. If $\Sigma$ is not $-1$-stable, there exists a non trivial vector subspace $U \subset k^n$, such that $A(U) \subset U$ and $\ker C \supset U$. It follows that $K(U) + L(U) \subset U \subset k^{n+p}$. But $-\dim U + \dim U =0$, which is not possible if $S$ is $1$-lowerstable.

The proof of the second statement is completely analogous, so we omit it.
\end{proof}

Notice that the previous theorem relies only on geometric invariant theory. We started with stability conditions on the space of linear dynamical systems $\tSigma_{n,m,p}$ given by the characters $1$ and $-1$ of the group $\GL_n$. The question asked was, whether the open subset of $1$-stable ($-1$-stable) systems $\Sigma$ was mapped under the morphism $\Phi_L: \tSigma_{n,m,p} \longrightarrow \tS_{n,m,p}$ in an open subset of $\chi$-stable systems for some $\chi \in [0,\infty]_{\Q}$. We answered this question in the affirmative and were able to identify the set of characters $\chi$ that fulfill the requirement.

We were interested in the $1$-stable ($-1$-stable) systems $\Sigma \in \tSigma_{n,m,p}$ because these are exactly the controllable (observable) systems. Lomadze already introduced a notion of controllability (observability) for Lomadze systems, and it comes as no surprise that the controllable (obserable) Lomadze systems are exactly the $1$-upperstable ($1$-lowerstable) Lomadze systems. This will be shown in the following proposition. Let us stress that the strategy adopted here serves two purposes: It allows us to realize the moduli space of Lomadze systems as an algebraic variety, which is by no means trivial, but also it shows how controllability  conditions can be generalized using geometric invariant theory. It might be interesting to know whether the other stability notions also admit a control theoretic interpretation.

\begin{proposition}[Okonek, Teleman]\label{identification}A Lomadze system $S \in \tL_{n,m,p}$ is 
\begin{enumerate}
\item controllable if and only if it is $1$-upperstable;
\item observable if and only if it is $1$-lowerstable.
\end{enumerate}
\end{proposition}

\begin{proof}
The first statement is \cite[Theorem 1.5]{okonek:04}. So let us prove the second one (the proof follows the same lines).
Assume $S$ to be $-1$-stable. Suppose $\rk \left[s K + t L\right] < n$ for some $(s, t) \neq 0 \in k^2$. Assume w.l.o.g. that $s =1$. Put $U:= \ker \left[s K + t L\right]$ and $V:= L(U)$. Then $K(U) \subset V$, since for $x \in U$ we have $K(x) + t L(x) =0 \in V$, $t L(x) \in V$, and hence $K(x) \in V$. But $-\dim U + \dim V \leq 0$, which contradicts the assumed $-1$-lowerstability.

We have to prove that if $\rk [s K + t L, M] < n+p$ for all $(s, t) \in k^2$, then there exist vector subspaces $U \subset k^n$ and $V \subset k^{n+p}$, with $K(U) + L(U) \subset V$, $\im M \subset V$, and such that $-(n - \dim U) + (n+p - \dim V) > 0$.
Put $W:= \im M$, choose a basis $k^{n+p}/W \cong k^{n+p-\dim W}$, and consider the induced maps
\[ \bar{K}, \bar{L}: k^n \longrightarrow k^{n+p}/W \cong k^{n+p - \dim W}.\]
By assumption $\rk [s \bar{K} + t \bar{L}] < n+p - \dim W$ for all $(s, t) \in k^2$. Hence we can apply a theorem of Gantmacher \cite[Theorem 12.3.4]{gantmacher:86}, which states that there exist linear automorphisms $g: k^n \iso k^n$ and $h: k^{n+p-\dim W} \iso k^{n+p -\dim W}$, such that 
\begin{equation} h \bar{K}g^{-1} = \begin{pmatrix} \bar{K}_\varepsilon & 0 \\ 0 & \ast \end{pmatrix} \text{ and }  
h \bar{L} g^{-1} = \begin{pmatrix} \bar{L}_\varepsilon & 0 \\ 0 & \ast \end{pmatrix}, \end{equation} 
where $\bar{K}_\varepsilon = \left(\begin{smallmatrix} 0 \\ id_{\varepsilon} \end{smallmatrix} \right)$ and $\bar{L}_{\varepsilon} = \left(\begin{smallmatrix} \id_{\varepsilon} \\ 0 \end{smallmatrix} \right)$ for some $\varepsilon < n$.  We assume w.l.o.g. that $\bar{K}$ and $\bar{L}$ are already in that form. Now put $U:= \langle e_{\varepsilon +1}, \ldots, e_n \rangle \subset k^n$ and $\tilde{V}:= \langle e_{\varepsilon + 2}, \ldots, e_{n+p - \dim W} \rangle \subset k^{n+p -\dim W}$. Let $V$ be the preimage of $\tilde{V}$ under the quotient map $k^{n+p} \longrightarrow k^{n+p}/W \cong k^{n+p - \dim W}$. Then we obtain:
\begin{enumerate}
\item $K(U) + L(U) \subset V$ and $\im M \subset V$;
\item $\dim V < \dim U  +p$.
\end{enumerate}
This yields the contradiction we were looking for.

Now assume $S$ to be observable. Choose vector subspaces $U \subset k^n$ and $V \subset k^{n+p}$, such that $K(U) + L(U) \subset V$. From observability it follows that $\rk [s K + t L] = n$ for all $(s, t) \neq 0 \in k^2$. Clearly $(s K + t L) (U) \subset V$ and hence $\dim U \leq \dim V$. Suppose $\dim U = \dim V$. Let $\varphi_{s, t}: U \longrightarrow V$ be the morphism induced by $s K + t L$. Since $k$ is algebraically closed, we have $\det \varphi_{s, t} = 0$ for some pair $(s, t) \neq 0$, and thus $\varphi_{s, t}$ is not injective. This is a contradiction and hence the first condition for $S$ being $1$-lowerstable is verified.

Suppose now that for some vector subspaces $U \subset k^n$ and $V\subset k^{n+p}$ we have $K(U) + L(U) \subset V$ and $\im M \subset V$. We know that $\rk[s K + t L, M] = n+p$ for some $(s, t) \in k^2$, i.e. 
\[\psi_{s, t} = [ s K + t L, M]: k^n \oplus k^{m +p} \longrightarrow k^{n+p} \]
is surjective. We have $\psi_{s, t}(U \oplus k^{m+p}) \subset V$ by assumption, so the morphism factors:
\[ \xymatrix{ k^n \oplus k^{m+p} \ar[r]^{\psi_{s, t}} \ar[d] & k^{n+p} \ar[d] \\
k^n/U \ar@{-->}[r]^{\overline{\psi}_{s, t}} & k^{n+p}/V } \]
Since $\overline{\psi}_{s, t}$ is surjective by construction, we conclude that $n - \dim U \geq n+p - \dim V$, which is equivalent to the inequality needed to complete the proof. 
\end{proof}

\begin{corollary}Let $S \in \tL_{n,m,p}$ be a Lomadke system. 
\begin{enumerate}
\item The system $S$ is controllable if and only if it is $\chi$-semistable or $\chi$-stable for any $\chi=(\chi_0, \chi_1) \in \Z_{\leq 0} \times \Z_{\geq 0}$ satisfying $(n-1)\chi_0 + n\chi_1 < 0$ and $\chi_0 + \chi_1 > 0$.
\item The system $S$ is observable if and only if it is $\chi$-semistable or $\chi$-stable for any $\chi=(\chi_0, \chi_1) \in \Z_{\leq 0} \times \Z_{\geq 0}$ satisfying $(n+1)\chi_0 + n\chi_1 > 0$ and $\chi_0 + \chi_1 < 0$.
\end{enumerate}
\end{corollary}
\begin{proof}The statement follows immediately from the observation that $\Lambda_1^+=(\frac{n-1}{n}, 1)$ and $\Lambda_1^- = (1, \frac{n+1}{n})$ considering that we identified a character $\chi=(\chi_0, \chi_1)$ with the point $- \frac{\chi_0}{\chi_1} \in [0, \infty]_{\Q}$.
\end{proof}

Since the respective subsets of controllable and observable Lomadze systems are given as the stable loci with respect to some character of $\GL_{n,n+p}$, we can define 
\begin{equation} L_{n,m,p}^c:= \tL_{n,m,p}^c//\GL_{n,n+p} \text{ and } L_{n,m,p}^o:= \tL_{n,m,p}^o//\GL_{n,n+p}.\end{equation}
By $L_{n,m,p}^{c, \reg}$ (respectively $L_{n,m,p}^{o,\reg}$) we denote the image of $\tL_{n,m,p}^{c} \cap \tL_{n,m,p}^{\reg}$ (respectively $\tL_{n,m,p}^{o} \cap \tL_{n,m,p}^{\reg}$) in the corresponding quotient space. The closed embedding $\Phi_L: \tSigma_{n,m,p} \hookrightarrow \tL_{n,m,p}$ induces isomorphisms
\begin{equation}
\Sigma_{n,m,p}^c \iso L_{n,m,p}^{c, \reg} \subset L_{n,m,p}^c \text{ and } \Sigma_{n,m,p}^o \iso L_{n,m,p}^{o, \reg} \subset L_{n,m,p}^o.
\end{equation} 

\begin{corollary}The quotient maps $\tL_{n,m,p}^c \longrightarrow L_{n,m,p}^c$ and $\tL_{n,m,p}^p \longrightarrow L_{n,m,p}^o$ are principal $\GL_{n,n+p}$-bundles. The quotients $L_{n,m,p}^c$ and $L_{n,m,p}^o$ are smooth, projective algebraic varieties of dimension $n(m+p)+pm$.
\end{corollary}

\begin{proof}The quotient maps are principal bundles by from Proposition \ref{stabilizer} and Proposition \ref{luna}. The quotients are projective by Corollary \ref{general_procesi}. 
\end{proof}

\subsection{The geometry of the Lomadze compactification}

In \cite{stromme:87}, Str\o mme describes a certain Quot scheme as the base space of a principal bundle. Ravi and Rosenthal observed in \cite{rosenthal:94} that this principal bundle coincides with the quotient map $\tL_{n,m,p}^c \longrightarrow L_{n,m,p}^c$. We will need this algebraic description in order to be able to give a precise algebraic description of the Helmke compactification and to compare the two compactifications. 
 
Fix natural numbers $r<q$ and $d$, as well as a $k$-vector space $V$ of dimension $q$. Put $P:=(T+1)r+d \in \Q[T]$ and $\Ec:= V \otimes \opone$. Let $R:=\Quot_{\pone/k}^{P,\Ec}$ be the Quot scheme parametrizing the quotients
\begin{equation} V \otimes \opone \longrightarrow \Qc \longrightarrow 0 \end{equation}
on $\pone$ of degree $d$ and rank $r$.\\
In \cite{stromme:87}, Str\o mme exhibited this quot scheme as the base space of a $\GL_{d,r+d}$-principal bundle
$X_0 \longrightarrow R$ with $X_0$ an open subset of the affine space 
\begin{equation} X = \Hom_{\opone}\bigl(k^d \otimes \opone(-1), k^{r+d} \otimes \opone \bigr) \times \Hom_k \bigl(V, k^{r+d} \bigr).\end{equation}
To be more precise, associate to an element $(\mu, \nu) \in X$ the diagram
\begin{equation} \xymatrix{ k^d \otimes \opone(-1)  \ar[r]^{\mu} & k^{r+d} \otimes \opone \ar[r] & \coker \mu \ar[r]& 0 \\
& V \otimes \opone \ar[u]_{\nu} \ar[ru]_{\tilde{\nu}} & } \end{equation}
Let $X_0$ be the open subset of pairs $(\mu, \nu)$ that verify
\begin{enumerate}
\item the morphism $\mu: k^d \otimes \opone(-1) \longrightarrow k^{r+d} \otimes \opone$ is an injective morphism of sheaves;
\item the induced map $\tilde{\nu} : V \otimes \opone \longrightarrow \coker \mu$ is surjective.
\end{enumerate}

The group $\GL_{d,r+d}$ acts on $X$ and $X_0$ by change of basis. By definition of $X_0$ there is a natural $\GL_{d,r+d}$-invariant morphism $g: X_0 \longrightarrow R$. Let us recall how it acts on $k$-rational points (see \cite{stromme:87} and \cite{diplomarbeit} for a more detailed discussion):
To a pair $(\mu, \nu) \in X_0$ we associate the point $[\tilde{\nu}: V \otimes \opone \longrightarrow \coker \mu \longrightarrow 0] \in R$. 

Let $p_{_R} : R \times \pone \longrightarrow R$ be the projection and consider the universal sequence on $R$ twisted by an integer $t \geq -1$: 
\begin{equation} 0 \longrightarrow \mathcal{A}(t) \longrightarrow V \otimes \mathcal{O}_{R \times \pone}(t) \longrightarrow \mathcal{B}(t) \longrightarrow 0. \end{equation}

The sheaves \begin{equation}\label{vb} \mathcal{B}_t:= (p_{_R})_\ast \mathcal{B}(t) \end{equation} are easily seen to be locally free of rank $(t+1)r +d$.
Denote the principal $\GL_{(t+1)r +d}$-bundles associated with $\Bc_t$ 
by $p_t: P_t \longrightarrow R$. Then
\begin{equation} p: P_{-1} \times_R P_0 \longrightarrow R \end{equation}
is a principal $\GL_d \times \GL_{r+d}$-bundle on $R$.

\begin{theorem}[Str\o mme]\label{quot}There is a canonical $GL_d \times \GL_{r+d}$-equivariant isomorphism 
\[ \xymatrix{ X_0 \ar[rr]^{\cong} \ar[dr]_g & & P_1 \times_R P_0 \ar[dl]^{p} \\ & R & }\]
of algebraic varieties over $R$. \\
In particular: 
\begin{enumerate}
\item $g: X_0 \longrightarrow R$ is a principal $\GL_d \times \GL_{r+d}$-bundle;
\item $R$ is a smooth, projective variety of dimension $q(r+d) -r^2$.
\end{enumerate}
\end{theorem}

\begin{proof}\cite[Theorem 2.1]{stromme:87}. 
\end{proof}

When $q=r+1$, this Quot scheme is a projective space.

\begin{proposition}\label{quot_proj}
Let $r,d \in \N_{\geq 0}$, and let $V$ be a $k$-vector space of dimension $q:=r+1$. Put $P:=(T+1)r+d \in \Q[T]$, $\Ec:= V \otimes \opone$, and let $R:=\Quot_{\pone / k}^{P, \Ec}$ be the Quot scheme parametrizing quotients
\[ V \otimes \opone \longrightarrow \Bc \longrightarrow 0 \]
of degree $d$ and rank $r$.
 There is a natural isomorphism
\[ R \iso \Grass_k\bigl(V \otimes H^0\bigl(\opone(d)\bigr), P(d)\bigr) \cong \proj{q(d+1) -1}. \]
\end{proposition}

\begin{proof}This is an adapted version of the construction explained e.g. in \cite{hl:97}. See \cite[Proposition 1.13]{diplomarbeit} for a detailed proof. A $k$-rational point $[q: V \otimes \opone \longrightarrow \Qc \longrightarrow 0]$ of $R$ is mapped by this isomorphism to the sequence
\[ \left[ 0 \longrightarrow H^0( \ker q (d)) \longrightarrow V \otimes H^0 \bigl( \opone(d) \bigr) \longrightarrow H^0(\Qc(d)) \longrightarrow 0 \right], \]
which defines a point in $\Grass_k \bigl( V \otimes H^0\bigl(\opone(d)\bigr), P(d) \bigr)$.
\end{proof}

\begin{lemma}\label{non_trivial}Let $r \geq 0, \, d>0$, and let $V$ be a $k$-vector space of dimension $q>r$. Then the vector bundle $\Bc_{-1}$ has non trivial Chern polynomial.
\end{lemma}

\begin{proof}
Let $\proj{d}$ be the projective space of lines in $H^0( \opone(d))$. Then the universal line bundle on $\proj{d}$ is given as a subbundle $0 \longrightarrow \Oc_{\proj{d}}(-1) \longrightarrow H^0(\opone(d)) \otimes \Oc_{\proj{d}}$. This inclusion induces a morphism of sheaves on $\proj{d} \times \pone$:
\[  \tilde{a}: \Oc_{\proj{d}}(-1) \boxtimes \opone(-d) \longrightarrow \left( H^0(\opone(d)) \otimes \Oc_{\proj{d}} \right) \boxtimes \opone(-d) \longrightarrow \Oc_{\proj{d} \times \pone}. \]
Choose a vector subspace $j: U \hookrightarrow V$ of codimension $r+1$ and a parametrized line $i: k \hookrightarrow V$ that intersects $U$ trivially. Consider the morphism of sheaves
\[ a: \Kc:= U \otimes \Oc_{\proj{d} \times \pone} \oplus \left( \Oc_{\proj{d}}(-1) \boxtimes \opone(-d) \right) \longrightarrow V \otimes \Oc_{\proj{d} \times \pone}\]
that is induced by $i \circ \tilde{a}$ and by $j$.
Let $p: \proj{d} \times \pone \longrightarrow \proj{d}$ and $q: \proj{d} \times \pone \longrightarrow \pone$ be the respective projections. By construction the morphism $a$ is injective on the fibers over $p$ and thus its cokernel $\Qc:= \coker a$ is flat over $\pone$ with Hilbert polynomial $(T+1)r +d$ on the fibers of $p$.

Let $\alpha: \proj{d} \longrightarrow R$ be the corresponding morphism to the Quot scheme $R$. By cohomology and base change, we have $\alpha^\ast \Bc_{-1} = p_\ast \Qc(-1) =: \Qc_{-1}$. Thus it suffices to prove the statement for $\Qc_{-1}$.

Twist the sequence $0 \longrightarrow \Kc \longrightarrow V \otimes \Oc_{\proj{d} \times \pone} \longrightarrow \Qc \longrightarrow 0$ by $-1$ (i.e. tensor it with $q^\ast \opone(-1)$) and the push it forward via $p$ to obtain the sequence
\[  0 \longrightarrow \Qc_{-1} \longrightarrow R^1 p_\ast \Kc(-1) \longrightarrow R^1 p_\ast V \otimes \Oc_{\proj{d} \times \pone}(-1)=0,\]
and hence an isomorphism $\Qc_{-1} \cong R^1 p_\ast \Kc(-1)$. Recall that $\Kc = U \otimes \Oc_{\proj{d} \times \pone} \oplus \left( \Oc_{\proj{d}}(-1) \boxtimes \opone(-d) \right)$. The statement now follows from the observation that $R^1p_\ast \Kc(-1) = H^1(\opone(-d-1)) \otimes \Oc_{\proj{d}}(-1) \cong H^0(\opone(d-1))^\vee \otimes \Oc_{\proj{d}}(-1)$.
\end{proof}

Now we specialise to the case of quotients of $V \otimes \opone \longrightarrow Q \longrightarrow 0$ of rank $p$ and degree $n$, where $V:=k^{m+p}$. There is a straightforward identification of the space 
\[X = \Hom_{\opone}\bigl(k^n \otimes \opone(-1), k^{n+p} \otimes \opone \bigr) \times \Hom_k \bigl(k^{m+p}, k^{n+p} \bigr)\] with the space of Lomadze systems.

\begin{proposition}[Ravi, Rosenthal]\label{identification_h_quot} There is a canonical $\GL_{n,n+p}$-equivariant isomorphism of algebraic varieties
\[ \tL_{n,m,p} \iso X = \Hom_{\opone}\bigl(k^n \otimes \opone(-1), k^{n+p} \otimes \opone \bigr) \times \Hom_k \bigl(k^{p+m}, k^{n+p} \bigr). \]
Under this isomorphism the open subset $\tL_{n,m,p}^c$ of controllable systems corresponds to the open subset $X_0$. \\
In particular, there exists a canonical isomorphism 
\[L_{n,m,p}^c \cong \Quot_{\pone/k}^{P, \Ec},\]
where $P=(T+1)p +n$ and $\Ec=k^{p+m} \otimes \opone$, i.e. $\Quot_{\pone/k}^{P, \Ec}$ parametrizes quotients 
\[ k^{p+m} \otimes \opone \longrightarrow \Qc \longrightarrow 0 \]
of rank $p$ and degree $n$.
\end{proposition} 
\begin{proof}
Given a linear dynamical system $(K,L,M)$, we put 
\[F:= Ks + Lt \in k[s,t]^{(n+p) \times n}.\]
 It corresponds to a morphism of graded $k[s, t]$-modules
\[ k[s, t](-1)^n \longrightarrow k[s,t]^{n+p} \]
and thus to an element $\tilde{F} \in \Hom_{\opone}\bigl(k^n \otimes \opone(-1), k^{n+p} \otimes \opone \bigr)$. \\
Furthermore $M \in k^{(n+p) \times (p+m)} \cong \Hom_k \bigl(k^{m+p}, k^{n+p} \bigr)$, so that $(\tilde{F}, M) \in X$. Clearly the assignment $(K,L,M) \mapsto (\tilde{F}, M)$ defines an isomorphism of $\GL_{n,n+p}$-varieties. \\
Now check that under this morphism the controllable systems are indeed mapped onto the open subset $X_0$ defined by Str\o mme.
\end{proof}

\begin{corollary}\label{S-Chow}The group underlying the Chow ring $A^\ast \left( L_{n,m,p}^c \right)$ is free abelian of rank
\[ \rk_{\Z} A^\ast \left( L_{n,m,p}^c \right) = \binom{m+p}{p} \binom{n +2m -1}{n}. \]
In the case $k=\C$, we have $A^k(L_{n,m,p}^c) \cong H^{2k}(L_{n,m,p}^c, \Z)$ and the odd cohomology vanishes. In particular, the topological Euler characteristic is $\chi_{top}(L_{n,m,p}^c) = \rk_{\Z} A^\ast \left( L_{n,m,p}^c \right)$.
\end{corollary}

\begin{proof}This follows from \cite[Corollary 1.4]{stromme:87} and the subsequent remark together with \cite[Theorem 3.5]{stromme:87}
\end{proof}

\begin{corollary}\label{L_systems_smooth}\
 The principal bundle $\tL_{n,m,p}^c \longrightarrow L_{n,m,p}^c$ is  non trivial when $n,m \geq 1$.
\end{corollary}

\begin{proof}Immediate consequence of Lemma \ref{non_trivial}.
\end{proof}

\begin{example}[Single Input Systems]\label{h_single_input}
\comment{
By Proposition \ref{quot_proj} and Proposition \ref{identification_h_quot} we have
\begin{enumerate}
\item $L_{n,1,p}^c \cong \proj{(p+1)(n+1)-1}$;
\item $L_{n,1}^c \cong \proj{n}$.
\end{enumerate}
Put $N:=(p+1)(n+1)-1$. For $t=-1,0$, the locally free sheaf $\Bc_t$ on $\proj{N}$ is given as an extension
\[ 0 \longrightarrow k^{p+1} \otimes H^1\bigl(\opone(t) \bigr) \otimes \Oc_{\proj{N}} \longrightarrow \Bc_t \longrightarrow R^1 {\pi_1}_\ast \Ac(t) \longrightarrow 0,\]
where $\pi_1: \proj{N} \times \pone \longrightarrow \proj{N}$ is the projection to the first factor.

Now observe that $\Ext^1_{\proj{N}} \left( R^1 {\pi_1}_\ast \Ac(t), k^{p+1} \otimes H^0 \bigl( \opone(t)\bigr) \right) =0$, and therefore we obtain 
\begin{itemize}
\item $\Bc_{-1} \cong H^0\bigl( \opone(n-1) \bigr)^\vee \otimes \Oc_{\proj{N}}(-1)$;
\item $\Bc_0 \cong \Bigl( H^0 \bigl( \opone(n-2) \bigr)^\vee \otimes \Oc_{\proj{N}}(-1) \Bigr) \oplus \Bigl( k^{p+1} \otimes \Oc_{\proj{N}} \Bigr)$.
\end{itemize}
}

This completely describes the principal bundles $\tL_{n,1,p}^c \longrightarrow L_{n,1,p}^c$: by Proposition \ref{identification_h_quot} we know that the principal bundle $\tL_{n,1,p}^c \longrightarrow L_{n,1,p}^c$ is isomorphic to the principal bundle $P_{-1} \times_R P_0 \longrightarrow R$, where $R=\proj{N}$. So we need to understand the two bundles $P_{-1} \longrightarrow R$ and $P_0 \longrightarrow R$. But these are defined to be the principal bundles of frames in the vector bundles $\Bc_{-1}$ and $\Bc_0$, respectively. In general, if $E \longrightarrow X$ is a vector bundle over a variety $X$ of rank $r$, then the principal $\GL_r$-bundle $P \longrightarrow X$ of frames in $E$ is given as follows: over a point $x \in X$ the fiber $P_x = \Isom(k^r, E_x)$ consists of the set of isomorphisms of $k^r$ with the fiber of the vector bundle $E$ over $x$. The group $\GL_r$ acts naturally on this space.
\end{example}

\section{The Helmke compactification}

The Helmke compactification generalizes the equations (\ref{lds}) to 
\begin{equation} Ex_{t+1} = Ax_t + Bu_t, \; \; Fy_t = Cx_t + Du_t \end{equation}
by introducing additional matrices $E \in k^{n\times n}$ and $F \in k^{p \times p}$.

\begin{definition}A Helmke system of type $(n,m,p)$ is a $6$-tuple $(E,A,B,C,D,F)$, consisting of matrices $E,A \in k^{n \times n}$, $ B \in k^{n \times m}$, $C \in k^{p \times n}$, $D \in k^{p \times m}$, $F \in k^{p \times p}$.
\end{definition}

We denote the vector space of all Helmke systems by $\tH_{n,m,p}$. The reductive linear algebraic group $\GL_{n,n,p}= \GL_n \times \GL_n \times \GL_p$ acts on $\tH_{n,m,p}$ as follows:
\begin{equation} (g_0, g_1, g_2) . (E,A,B,C,D,F) = (g_1 E g_0^{-1}, g_1 A g_0^{-1}, g_1 B, g_2C g_0^{-1}, g_2 D, g_2F) \end{equation}
 for  $(g_0,g_1,g_2) \in \GL_{n,n,p} \text{ and } (E,A,B,C,D,F) \in \tH_{n,m,p}$.

\begin{definition}A Helmke system $H=(E,A,B,C,D,F) \in \tH_{n,m,p}$ is called controllable if it verifies the following three conditions:
\begin{enumerate}
\item $\det [ sE + tA ] \neq 0 \text{ for some } (s,t) \in k^2$;
\item $\rk [ sE + tA, B ] = n \text{ for all } (s,t) \in k^2 \wo \{ 0 \}$;
\item $\rk [F,C,D] = p.$
\end{enumerate}
\end{definition}

There is a natural map $\Phi_H: \tSigma_{n,m,p} \longrightarrow \tH_{n,m,p}, \, \Sigma=(A,B,C,D) \mapsto H=(\id_n, A,B,C,D, \id_p)$. Furthermore, there is a natural inclusion $\varphi_H: \GL_n \longrightarrow \GL_{n,n,p}, \, g_0 \mapsto (g_0, g_0, \id_p)$.

\begin{lemma}The map $\Phi_H: \tSigma_{n,m,p} \longrightarrow \tH_{n,m,p}$ is a closed $\varphi_H$-equivariant embedding.
\end{lemma}
\begin{proof}
This is a straightforward calculation.
\end{proof}

\begin{proposition}\label{H-stabilizer}Let $H=(E,A,B,C,D,F) \in \tH_{n,m,p}$ be a controllable Helmke system. Then the stabilizer of the $\GL_{n,n,p}$-action at the point $H$ is trivial.
\end{proposition}

\begin{proof}Let $H=(E,A,B,C,D,F)$ be controllable and let $g=(g_0, g_1, g_2) \in \GL_{n,n,p}$ be an element in the stabilizer of $H$. Then clearly $S:= (E,A,B) \in \tL_{n,m}$ is a controllable Lomadze system and $(g_0, g_1) \in \GL_{n,n}$ is an element in the stabilizer of $S$. It follows from  Proposition \ref{stabilizer} that $g_0=g_1 = \id_n$. Therefore $gH=(E,A,B, g_2 C, g_2D, g_2F)$. By assumption the rank of the matrix $[C,D,F]$ is $p$. Therefore $g_2=\id_p$.
\end{proof}

\begin{lemma}Let $\Sigma \in \tSigma_{n,m,p}$ be a linear dynamical system. It is controllable if and only if $\Phi_H(\Sigma) \in \tH_{n,m,p}$ is controllable.
\end{lemma}
\begin{proof}By definition a system $\Sigma=(A,B,C,D) \in \tSigma_{n,m,p}$ is controllable if and only if $(A,B) \in \tSigma_{n,m}$ is controllable. The system $(A,B)$ is controllable if and only if $\Phi_L(A,B)=(\id_n,A,B) \in \tL_{n,m}$ is controllable. Therefore $(\id_n, A,B,C,D,\id_p)$ verifies the first two controllabilty conditions if and only if $\Sigma$ is controllable. Since the third one is obviously verified, the statement follows.
\end{proof} 

\begin{remark} Let $Q$ be the following quiver:
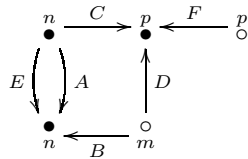
\[ \xymatrix{ \stackrel{n}{\bullet} \ar@/^/[d]^{A} \ar@/_/[d]_{E} \ar[r]^{C} & \stackrel{p}{\bullet} &  \stackrel{p}{\circ} \ar[l]_{F} \\
\underset{n}{\bullet} & \underset{m}{\circ} \ar[l]^{B} \ar[u]_{D}} \]

The space of all Helmke systems $\tH_{n,m,p}$ is the space of  representations of this quiver of the indicated dimension.
\end{remark}

\begin{proposition}\label{s_stability}Let $(r,s,t) \in \Z^3$ with $nr + (n-1)s + \min\{p,n\}t < 0, \, s+r > 0$, and $t>0$. Let $H \in \tH_{n,m,p}$ be a Helmke system. The following statements are equivalent:
\begin{enumerate}
\item $H$ is controllable;
\item $H$ is $(r,s,t)$-stable;
\item $H$ is $(r,s,t)$-semistable.
\end{enumerate}
\end{proposition}

\begin{proof}This is \cite[Theorem 3.50]{diplomarbeit}. We omit the proof.
\end{proof}

This statement shows that the space of controllable Helmke systems is the set of (semi)stable points with respect to some character $\chi=(\chi_0,\chi_1, \chi_2) \in \Z^3$. Thus we may write
\begin{equation} H_{n,m,p}^c:= \tH_{n,m,p}^c//\GL_{n,n,p}.\end{equation}
The morphism $\Phi: \tSigma_{n,m,p}^c \longrightarrow \tH_{n,m,p}^c$ described above descends to an open immersion
\[ \Sigma_{n,m,p}^c \hookrightarrow  H_{n,m,p}^c\]
with image all orbits of systems $(E,A,B,C,D,F)$ with $\det E \neq 0$ and $\det F \neq 0$.

\begin{corollary}The quotient space $H_{n,m,p}^c$ is a smooth projective algebraic variety of dimension $n(m+p) + pm$. The quotient map $\tH_{n,m,p}^c \longrightarrow \tH_{n,m,p}^c$ is a principal $\GL_{n,n,p}$-bundle.
\end{corollary}

\begin{proof}The quotient is projective as a direct consequence of Corollary \ref{general_procesi}. From Proposition \ref{luna} and Proposition \ref{H-stabilizer} it follows that the quotient maps are principal bundles
\end{proof}

There is a natural forgetful map $\tH_{n,m,p}^c \longrightarrow \tH_{n,m}^c,$
that descends to a map on the quotients. Helmke has proved that in the analytic category this map makes $H_{n,m,p}^c$ into a Grassmann bundle over the space $H_{n,m}^c$ of controllable Helmke systems without output. In Corollary \ref{ppties_implicit_moduli} we will prove that it is an algebraic Grassmann bundle. In the case $p=0$ the Helmke and the Lomadze compactifications clearly coincide: $H_{n,m}^c = L_{n,m}^c$. Furthermore we  had identified the latter space with a Quot scheme. Str\o mme has obtained several results on the geometry of this Quot scheme. In order to calculate invariants of the space $H_{n,m,p}^c$ we identify the Grassmann bundle $H_{n,m,p}^c \longrightarrow H_{n,m}^c$ as a Grassmann bundle of subspaces in a vector bundle on this quot scheme.

The first step is the following technical result that we cite here for completeness.

\begin{proposition}\label{principal_grassmann}Let $G$ be a linear algebraic group. Let $d,N \in \N$ with $0 \leq d<N$. Let $p: P \longrightarrow X$ be a principal $G$-bundle, and let $\lambda: G \longrightarrow \GL_N$ be a representation of $G$. Let $G$ act on $\Grass_k(d,k^N)$ by
\[ G \times \Grass_k(d,k^N) \longrightarrow \Grass_k(d,k^N), \, (g,U) \mapsto \lambda(g) \bigl(U\bigr).\]
The associated fiber bundle $P \times^G \Grass_k(d,k^N) \longrightarrow X$ is a Grassmann bundle in the Zariski topology of type $(d, k^N)$ on X. There is a canonical isomorphism
\[ P \times ^G \Grass_k(d,k^N) \iso \Grass_X(d, P \times^G k^N).\]
\end{proposition}

\begin{proof}\cite[Proposition 1.30]{diplomarbeit}.
\end{proof}

\begin{theorem}\label{ppties_implicit_moduli}\ 
\begin{enumerate}
\item $\tH_{n,m} = \tL_{n,m}$ and $\tH_{n,m}^c = \tL_{n,m}^c$.  The quotient map $\tH_{n,m}^c \longrightarrow H_{n,m}^c$ is a principal $\GL_{n,n}$-bundle and there is a canonical isomorphism \[f: H_{n,m}^c \iso \Quot_{\opone/k}^{\Ec, P},\]
where $\Ec:=k^m \otimes \opone$, and $P:=n \in \Q[T]$.
\item The morphism $\Psi: \tH_{n,m,p}^c \longrightarrow \tH_{n,m}^c, \, (E,A,B,C,D,F) \mapsto (E,A,B)$ descends to a map
\[ \psi: H_{n,m,p}^c \longrightarrow H_{n,m}^c.\]
It is a Grassmann bundle of type $(p, k^{n+m+p})$, canonically isomorphic to the Grassmann bundle of $p$-planes in the vector bundle 
\[\Ec_{n,m} \oplus \Oc^{\oplus (p+m)}.\] 
Here $\Ec_{n,m} = \tH_{n,m}^c \times^{\GL_{n,n}} \aff{n}$ is obtained via the representation $\GL_{n,n} \longrightarrow \GL_{n}, (h,g) \mapsto {(h^{-1})}^\tau$.
Furthermore
\[ f^\ast \Bc_{-1}^\vee \cong \Ec_{n,m}, \]
where $\Bc_{-1}$ is the  vector bundle introduced in the previous section.
\end{enumerate}
\end{theorem}

\begin{proof}\ 
The first statement is Proposition \ref{identification_h_quot}. So let us prove the second statement.
We write $\Hom(k^{n+m+p}, k^p)^s \subset \Hom(k^{n+m+p}, k^p)$ for the open subset of surjective maps. Recall that the Grassmannian of $p$-quotients of $k^{n+m+p}$ $\Grass(k^{n+m+p},p)$ can be defined as the quotient of $\Hom(k^{n+m+p}, k^p)$ under the natural left action of $\GL_p$ that is given by change of basis in the target space.
Define a morphism
\[ \alpha: \tH_{n,m,p}^c \longrightarrow \tH_{n,m}^c \times \Hom(k^{n+m+p}, k^p)^s\]
by sending $(E,A,B,C,D,F)$ to the pair $\left( (E,A,B), [C,D,F] \right)$. We equip the target space with the following action of $\GL_{n,n,p}$
\begin{equation} (g_0, g_1, g_2) \left( (E,A,B), M \right) \mapsto \left( (g_1 E g_0^{-1}, g_1 A g_0^{-1}, g_1 B), g_2 M \tilde{\lambda}(g)^{-1} \right), \end{equation}
for $g=(g_0,g_1,g_2) \in \GL_{n,n,p}$, $(E,A,B) \in \tH_{n,m}^c$, $M \in \Hom(k^{n+m+p}, k^p)^s$, and 
where $\tilde{\lambda}: \GL_{n,n,p} \longrightarrow \GL_{n+m+p}, \; (g_0, g_1, g_2) \mapsto g_0 \oplus \id_{m+p}$.
The morphism $\alpha$ is a $\GL_{n,n,p}$-equivariant isomorphism.
Since on the target space the group $\GL_p \subset \GL_{n,n,p}$ acts only on $\Hom(k^{n+m+p}, k^p)^s$, we may first consider the quotient with respect to that action only: it is $\tH_{n,m}^c \times \Grass(k^{n+m+p}, p)$. Taking the quotient with respect to the induced $\GL_{n,n}$-action we obtain a $\GL_{n,n,p}$-invariant morphism
\[ \tH_{n,m,p}^c \longrightarrow \left( \tH_{n,m}^c \times \Grass(k^{n+m+p}, p) \right)// \GL_{n,n} \]
that descends to an isomorphism
\[ H_{n,m,p}^c \iso \left( \tH_{n,m}^c \times \Grass(k^{n+m+p}, p) \right)// \GL_{n,n}. \]
First notice that there is a natural isomorphism $\Grass(k^{n+m+p}, p) \cong \Grass(p, k^{n+m+p})$ that is given by dualizing the maps and identifying $k^{n+m+p}$ with its dual using the canonical basis. The induced action of $\GL_{n,n}$ on $\tH_{n,m}^c \times \Grass(p, k^{n+m+p})$ is given by
\begin{equation} g . (H, U) = (gH, {\tilde{\lambda}(g)^{-1}}^\tau (U)) \text{ for } g \in \GL_{n,n} \text{ and } (H,U) \in \tH_{n,m}^c \times \Grass(p, k^{n+m+p}).\end{equation}
The quotient with respect to this action is the fiber bundle $\tH_{n,m}^c \times_{\lambda} \Grass(k^{n+m+p}, p)$ associated with the principal $\GL_{n,n}$-bundle $\tH_{n,m}^c \longrightarrow H_{n,m}$ and the representation $\lambda: \GL_{n,m} \longrightarrow \GL_{n+m+p}, \, g \mapsto {\tilde{\lambda}(g^{-1})}^\tau$. 
We end up with an isomorphism
\[ H_{n,m,p}^c \iso \tH_{n,m}^c \times_{ \lambda} \Grass(p, k^{n+m+p}).\]

The preceeding proposition tells us that this Grassmann bundle is the Grassmann bundle of $p$-planes in the vector bundle $\tH_{n,m}^c \times_{\lambda} k^{n+m+p}$. If follows from the special form of the representation $\lambda$ that this vector bundle splits as $\Ec_{n,m} \oplus \Oc^{\oplus{(p+m)}}$, where $\Ec_{n,m} = \tH_{n,m}^c \times_{\delta} k^n$ with $\delta : \GL_{n,n} \longrightarrow \GL_n, \, (g_0, g_1) \mapsto (g_0^{-1})^\tau$. The remaining statement is an immediate consequence of the identification $H^c_{n,m} \cong \Quot_{\opone/k}^{\Ec, P}$ and of the description of the Quot scheme as a principal bundle.
\end{proof}

For the convenience of the reader, we recall the standard results on the Chow ring of Grassmann (projective) bundles.

\begin{theorem}\label{chow_bundles_grassmann}
Let $\pi: E \longrightarrow X$ be a smooth vector bundle of rank $r$.  
\begin{enumerate}  
\item Let $\Pb(\pi): \Pb(E) \longrightarrow X$ be the projective bundle associated with $E$. There is an isomorphism of  
graded rings
\[ A^\ast(\Pb(E)) \cong A^\ast(X)[T]/ \langle T^r + c_1(E)T^{r-1} + \ldots + c_r(E) \rangle,\]
induced by $T \mapsto c_1(\Oc_{E}(1))$.
In particular $A^\ast(\Pb(E))$ is a free $A^\ast(X)$-module of rank $r$.
\item Let $g: \Grass_X(d, E) \longrightarrow X$  be the Grassmann bundle of $d$-planes in $E$ and let
\[ 0 \longrightarrow \mathcal{K} \longrightarrow E_G \longrightarrow \mathcal{Q} \longrightarrow 0 \]
be the universal sequence on $G$. 
There is an isomorphism of graded rings
\[ A^\ast(\Grass_X(d,E)) \cong A^\ast(X)[X_1, \ldots, X_d, Y_1, \ldots, Y_{r-d}] \bigr/ \bigl \langle \textstyle\sum_{i=0}^{k} X_iY_{k-i} - c_k(E) \, | \, k=0,\ldots,r \bigr \rangle, \]
given by $X_i \mapsto c_i(\mathcal{K})$ and $Y_j \mapsto c_j(\mathcal{Q})$, where the polymial ring has the grading corresponding to the weight $(1, \ldots, d,1 \ldots, r-d)$, $X_0=Y_0=1$, and $X_i=Y_j=0$, $i>d,\, j>r-d$. 
In particular $A^\ast(G)$ is a free $A^\ast(X)$-module of rank ${r \choose d}$.
\end{enumerate}
\end{theorem} 

\begin{proof}
The first statement is \cite[Example 8.3.4]{mfk:94}, the second is \cite[Example 14.6.6]{mfk:94}. A proof of the second statement is given in \cite[Proposition 5.4]{jouanolou:77}.
\end{proof}

\begin{corollary}\label{H-Chow}The group underlying the Chow ring $A^\ast(H_{n,m,p}^c)$ is free abelian of rank 
\[ \rk_{\Z} A^\ast(H_{n,m,p}^c) = \binom{n+p+m}{p} \binom{n + 2m -1}{n}. \]
In the case $k=\C$, we have $A^k(H_{n,m,p}^c) \cong H^{2k}(H_{n,m,p}^c,\Z)$ and therefore
\[ \chi_{top} A^\ast(H_{n,m,p}^c) = \rk_{\Z} A^\ast(H_{n,m,p}^c). \]
\end{corollary}

\begin{proof}This is a consequence of Corollary \ref{S-Chow} and Theorem \ref{chow_bundles_grassmann}.
\end{proof}

\begin{corollary}For $p>0$, the compactifications $L_{n,m,p}^c$ and  $H_{n,m,p}$ are not isomorphic. If $k=\C$, they are not homeomorphic.
\end{corollary}

\begin{proof}The Chow group (Euler-Poincar\'e Characteristic) is an algebraic (topological) invariant. The statement is therefore a consequence of Corollary \ref{S-Chow} and Corollary \ref{H-Chow}.
\end{proof}

\begin{corollary}\label{chow_ring_standard}There are canonical isomorphisms of graded rings 
\begin{enumerate}
\item $A^\ast (H_{n,m,p}^c) \cong A^\ast(H_{n,m}^c)[X_1, \ldots, X_p, Y_1, \ldots, Y_{n+m}]/ I$, with \\
$I =  \langle \sum_{i=0}^j X_i Y_{j-i} + (-1)^{j+1} c_j (\Bc_{-1}) \vbar j =0, \ldots,n+m+p \rangle$, and $X_0=Y_0=1$, $X_i=Y_j=0$ for $i>p$ and $j>n+m$.
\item $A^\ast (H_{n,m,1}^c) \cong A^\ast(H_{n,m}^c)[T] / I$, with \\
$ I = \langle T^{n+m+1} - c_1( \Bc_{-1})T^{n+m} + c_2(\Bc_{-1})T^{n+m-1} + \ldots + (-1)^n c_n( \Bc_{-1})T^{m} \rangle$.
\end{enumerate}
\end{corollary}

\comment{
\begin{proof}This follows from Lemma \ref{ppties_implicit_moduli} and Theorem \ref{chow_bundles_grassmann}.
\end{proof}
}

\begin{example}[Single Input Systems]
We have already seen in Example \ref{h_single_input} that   
\[ H_{n,1}^c = L_{n,1}^c \cong \proj{n}.\]
By the same example we know that $\Bc_{-1} = H^0\bigl( \pone, \opone(n-1)\bigr)^\vee \otimes \Oc_{\proj{n}}(-1)$. In particular $c_t( \Bc_{-1}) = (1-h)^n$, where $h:= c_1 ( \opone(1))$.
Hence:  
\begin{equation*}\begin{split}
 A^\ast(H_{n,1,p}^c) \cong A^\ast( \proj{n})[X_1, \ldots, X_p, Y_1, \ldots, Y_{n+1}]/ I, \text{ where }\\
I =  \langle \textstyle\sum_{i=0}^j X_i Y_{j-i} + (-1)^j c_j(\Bc_{-1}) \vbar j=0, \ldots, n+p\rangle.
\end{split}\end{equation*}
For example:
\begin{itemize}  
\item $A^\ast(H_{1,1,2}^c) \cong \Z[h,X_1, X_2,Y_1,Y_2] / I$, where  \\
$ I = \langle h^2, \; Y_1 + X_1 +h, \; Y_2 + X_1Y_1 + X_2, \; X_1 Y_2 + X_2 Y_1 \rangle$;
\item $A^\ast(H_{1,1,1}^c) \cong \Z[h,T] / \langle h^{2}, T^{3} + h T^2 \rangle$. 
\end{itemize}
In particular, we see that the Grassmann bundle $H_{n,m,p}^c \longrightarrow H_{n,m}^c$ is in general non trivial.
\end{example}


\end{document}